\documentclass[reqno]{amsart}
\usepackage{amssymb,latexsym}
\usepackage{amsmath}
\usepackage{amsthm}
\usepackage{graphicx}
\usepackage{hyperref}
\usepackage{titletoc}
\numberwithin{equation}{section}
\newtheorem{theorem}{Theorem}[section]
\newtheorem{prop}[theorem]{Proposition}
\newtheorem{lemma}[theorem]{Lemma}

\newtheorem{definition}[theorem]{Definition}

\theoremstyle{definition}

\newcommand{\Z}{\mathbb Z}
\newcommand{\E}{\mathbb E}
\newcommand{\R}{\mathbb R}
\newcommand{\N}{\mathbb N}
\newcommand{\hm}{\mathbf{m}}
\newcommand{\Mp}{\mathbb{M}_p}
\newcommand{\Zm}{\mathcal{Z}}
\newcommand{\Zmm}{\mathcal{Z}^-}

\newcommand{\diam}{\textrm{diam}}
\newcommand{\Rtheta}{\theta^-}
\newtheorem{remark}{Remark}[section]

\newcommand{\ssubset}{\subset\subset}
\newcommand{\Capa}{\text{Cap}}
\newcommand{\BCap}{\text{BCap}}
\newcommand{\EsC}{\text{Es}}
\newcommand{\Esc}{\text{Esc}}
\newcommand{\es}{\text{ES}}
\newcommand{\KRW}{\mathbf k}
\newcommand{\SRW}{\mathbf s}
\newcommand{\BRW}{\mathbf b}
\newcommand{\Snake}{\mathcal S}
\newcommand{\bb}{\Gamma}
\newcommand{\rp}{\widetilde{r}}

\newcommand{\pS}{\mathbf p}
\newcommand{\rS}{\mathbf r}
\newcommand{\qS}{\mathbf q}
\newcommand{\qrS}{\mathbf q^-}
\newcommand{\PS}{\overline{\mathbf p}}
\newcommand{\RS}{\overline{\mathbf r}}
\newcommand{\DeFine}{\doteq}
\newcommand{\dist}{\rho}

\newcommand{\Ball}{\mathcal C}
\newcommand{\BallE}{\mathcal B}

\newcommand{\Rad}{\text{Rad}}
\newcommand{\Hm}{\mathcal{H}}

\begin{document}
\title[On CBRW I]{On the critical branching random walk I: Branching capacity and visiting probability}
\author{Qingsan Zhu}
\address{~Department of Mathematics, University of British
Columbia,
Vancouver, BC V6T 1Z2, Canada}
\email{qszhu@math.ubc.ca}
\date{}
\maketitle

\begin{abstract}
We extend the theory of discrete capacity to critical branching random walk. We introduce branching capacity for any finite subset of $\Z^d, d\geq5$. Analogous to the regular discrete capacity, branching capacity is closely related to the asymptotics of the probability of visiting a fixed finite set by a critical branching random walk starting from a distant point and the conditional distribution of the hitting point.
\end{abstract}

\section{Introduction}
The theory of discrete capacity, or the discrete potential theory, plays an essential role in the study of random walks (see e.g. \cite{L91, LL10, S76}). Let us first review some results on regular (discrete) capacity. For any finite subset $K$ of $\Z^d, d\geq3$, the escape probability $\es_K(x)$ is defined to be the probability that a random walk starting from $x\in\Z^d$ with symmetric jump distribution, denoted by $S_x=(S_x(n))_{n\in\N}$, never returns to $K$. The capacity of $K$, $\Capa(K)$ is given by:
$$
\Capa(K)=\sum_{a\in K}\es_K(a).
$$
We have
$$
\lim_{x\rightarrow \infty}\|x\|^{d-2}\cdot P(S_x \text{ visits } K)= a_d\Capa(K),
$$
where $a_d=\frac{1}{2d^{(d-2)/2}\sqrt{\det Q}}\Gamma(\frac{d-2}{2})\pi^{-d/2}$ and $\|x\|=\sqrt{x\cdot Q^{-1}x}/\sqrt{d}$ with $Q$ being the covariance matrix of the jump distribution.
Moreover, let $\tau_K=\inf\{n\geq 1:S_x(n)\in K\}$, then for any $a\in K$, we have
$$
\lim_{x\rightarrow \infty}P(S_x(\tau_K)=a|S_x \text{ visits } K)=\es_K(a)/\Capa(K).
$$
$\es_K(a)$ is usually called the equilibrium measure and the normalized measure $\es_K(a)/\Capa(K)$ is called the harmonic measure of set $K$. In fact, not only the distribution of the first visiting point, but also that of the last visiting point, conditioned on visiting $K$, converge to the same measure:
$$
\lim_{x\rightarrow \infty}P(S_x(\xi_K)=a|S_x \text{ visits } K)=\es_K(a)/\Capa(K),
$$
where $\xi_K=\sup\{n\geq 1:S_x(n)\in K\}$.

The results above apply to any symmetric irreducible jump distribution with some finite moment assumption. Unfortunately we do not find any reference for the nonsymmetric walks. However the following result is well-known and can be proved similarly to the symmetric case (see the Preface of \cite{LL10}). When the jump distribution is irreducible, nonsymmetric, with mean zero and, for simplicity, finite range, we have:
$$
\lim_{x\rightarrow \infty}P(S_x(\tau_K)=a|S_x \text{ visits } K)=\es^-_K(a)/\Capa(K),
$$
$$
\lim_{x\rightarrow \infty}P(S_x(\xi_K)=a|S_x \text{ visits } K)=\es_K(a)/\Capa(K),
$$
where $\es^-$ is the escape probability for the reversed random walk.

In this and subsequent papers \cite{Z162,Z163}, we study critical branching random walks. This paper and \cite{Z162} address the supercritical case ($d\geq5$) and \cite{Z163} the critical case ($d=4$). In this paper, we introduce branching capacity and construct analogous results for critical branching random walks (starting with a single initial particle) in $\Z^d,d\geq 5$. For a given distribution $\mu$ on $\N$ and a given distribution $\theta$ on $\Z^d$, we are interested in the branching random walk with offspring distribution $\mu$ and jump distribution $\theta$. We always assume (unless otherwise specified) in this paper that $d\geq 5$ and
\begin{itemize}
\item $\mu$ is a distribution on $\N$ with mean one and finite variance $\sigma^2>0$;
\item $\theta$ is a distribution on $\Z^d$ with mean zero not supported on a strict subgroup of $\Z^d$ and 'weak' $L^d$ in the following sense: there exists $C>0$, such that for any $r\geq 1$,
    \begin{equation}\label{as1}
    \theta(\{x\in \Z^d:|x|>r\})< C\cdot r^{-d}.
    \end{equation}
\end{itemize}
Note that \eqref{as1} holds if $\theta$ has finite $d$-th moments.

To extend the random walk results stated above to branching random walks, one needs to introduce analogues of the  escape probability. For a finite set $K$ of $\Z^d$, one could consider the probability that the branching random walk starting at $x$ (with offspring distribution $\mu$ and jump distribution $\theta$), denoted by $\Snake_x$, avoids $K$. However, this turns out not to be the right generalization. Two different extensions of the escape probability need to be defined: one for the first and one for the last visiting point of $K.$ We denote these by $\EsC_K(x)$ and $\Esc_K(x)$. Both correspond to infinite versions of the branching process. We defer the complete definitions to Section 3. We define the branching capacity of $K$ by
$$
\BCap(K)=\sum_{z\in K}\EsC_K(z)=\sum_{z\in K}\Esc_K(z),
$$
The equality of these two sums can be seen from the following.
\begin{theorem}\label{MT1}
For any nonempty finite subset $K$ of $\Z^d$ and $a\in K$, we have
\begin{equation}\label{m1-1}
\lim_{x\rightarrow \infty}\|x\|^{d-2}\cdot P(\Snake_x \text{ visits } K)= a_d\BCap(K);
\end{equation}
\begin{equation}
\lim_{x\rightarrow \infty}P(\Snake_x(\tau_K)=a|\Snake_x \text{ visits } K)=\EsC_K(a)/\BCap(K),
\end{equation}
\begin{equation}\label{m1-3}
\lim_{x\rightarrow \infty}P(\Snake_x(\xi_K)=a|\Snake_x \text{ visits } K)=\Esc_K(a)/\BCap(K),
\end{equation}
where $\tau_K$ and $\xi_K$ respectively are the first, and last respectively, visiting time of $K$ in a Depth-First search and $a_d$ is the same constant as in the random walk case.
\end{theorem}

Let us make some comments here. First, if $\mu$ is the degenerate measure with $\mu(1)=1$, then the branching random walk is just the regular random walk, and $\EsC_K$ ($\Esc_K$ respectively) is just $\es^-_K$ ($\es_K$ respectively). In this case,
Theorem \ref{MT1} is the classical result for random walk (for a more detailed discussion, see Remark \ref{compare}).

Second, this result tells us that conditioned on visiting a fixed set, the 'first' (or the last) visiting point converges in distribution. It turns out that we can say more about this. In fact, we also show (see Theorem \ref{T-hm}) that conditioned on visiting $K$, the set of entering points converges in distribution. Since the distribution of the intersection between $K$ and the range of $\Snake_x$ can be determined by the entering points, we have
\begin{theorem}\label{M2}
Conditioned on $\Snake_x$ visiting $K$, the intersection between $K$ and the range of $\Snake_x$ converges in distribution, as $x\rightarrow\infty$.
\end{theorem}

Third, this result gives the asymptotic behavior of the probability of visiting a fixed finite set by critical branching random walk starting from far away (for dimension $d\geq5$). In \cite{LL14}, Le Gall and Lin establish the following result (in our notation) in the subcritical dimensions $d\leq 3$ :
$$
\lim_{x\rightarrow\infty}\|x\|^{2}\cdot P(\Snake_x \text{ visits } 0)=\frac{2(4-d)}{d\sigma ^2}.
$$
They raise the question about the asymptotic of the probability of visiting a distant point in other dimensions ($d\geq 4$). Theorem \ref{MT1} solves this question for supercritical dimensions ($d\geq5$). The critical dimension $d=4$ is addressed in \cite{Z163} where we establish:
$$
\lim_{x\rightarrow\infty}\|x\|^{2}\log\|x\|\cdot P(\Snake_x \text{ visits } K)=\frac{1}{2\sigma ^2},
$$
and for any $a\in K$,
$$
\lim_{x\rightarrow\infty}P(\Snake_x(\tau_K=a)|\Snake_x \text{ visits } K)\;\text{exists},
$$
$$
\lim_{x\rightarrow\infty}P(\Snake_x(\xi_K=a)|\Snake_x \text{ visits } K)\;\text{exists}.
$$

We mention some related results here. In \cite{LL141} and \cite{LL14}, Le Gall and Lin establish the asymptotic of the range of a critical branching random walk conditioned on total size being $n$. In the supercritical dimensions, the range divided by $n$ converges in probability to a constant which they interprets as some escape probability. This constant is just $\BCap(\{0\})$ in our notation. For the range in the critical dimension, one can see \cite{LL141,Z163}.

We also construct the following bounds for the visiting probability by critical branching random walk when the distance $\dist(x,A)$ between $x$ and $A$ is not too small, compared with the diameter of $\diam(A)$.
\begin{theorem}\label{bd-finite}
For any finite $A\subseteq\Z^d$ and $x\in \Z^d$ with $\dist(x, A)\geq 0.1 \diam(A)$, we have:
\begin{equation}\label{bd-p}
P(\Snake_x \text{ visits } A)\asymp \frac{\BCap(A)}{(\dist(x,A))^{d-2}},
\end{equation}
where $f(x,A)\asymp g(x,A)$ indicates that there exists positive constants $c_1,c_2$ independent of $x,A$ such that $c_1f(x,A)\leq g(x,A)\leq c_2f(x,A)$.
\end{theorem}

One might compare this with the corresponding result for random walk:
\begin{equation}\label{SRW-visiting}
P(S_x  \text{ hits }A)\asymp \frac{\Capa(A)}{(\dist(x,A))^{d-2}}.
\end{equation}

Similarly to random walk, computing escape probabilities can be very difficult. Hence it might not be practical to estimate the branching capacity directly from the definition. However we can use \eqref{m1-1} in reverse: by estimating the probability of visiting a set, we can give bounds for the branching capacity of that set. We do so in Section 10 and find the order of the magnitude of the branching capacity of low dimensional balls:
\begin{theorem}\label{MT2}
Let $B^m(r)$ be the $m$-dimensional balls with radius $r$ (as a subset of $\Z^d$), i.e. $\{z=(z_1,0)\in\Z^m\times\Z^{d-m}=\Z^d: |z_1|\leq r\}$. For any $r>2$, we have:
\begin{equation}
\BCap(B^m(r))\asymp \left\{
\begin{array}{ccc}
r^{d-4}&\text{ if }& m\geq d-3;\\
r^{d-4}/\log r &\text{ if }& m=d-4;\\
r^{m} &\text{ if }& m\leq d-5.\\
\end{array}
\right.
\end{equation}
\end{theorem}
One might compare this with the corresponding result about regular capacity:
\begin{equation}\label{capa-B}
\Capa(B^m(r))\asymp \left\{
\begin{array}{ccc}
r^{d-2}&\text{ if }& m\geq d-1;\\
r^{d-2}/\log r &\text{ if }& m=d-2;\\
r^{m} &\text{ if }& m\leq d-3.\\
\end{array}
\right.
\end{equation}

Our definition of branching capacity depends on the offspring distribution $\mu$ and jump distribution $\theta$. From the previous result, one can see that branching capacities of a ball for different $\mu$'s and $\theta$'s are comparable. We prove the following:
\begin{theorem}\label{MT5}
Suppose that $\mu_1,\mu_2$ are two nondegenerate critical offspring distributions with finite second moment and write $\BCap_{\mu_1,\theta}$ and $\BCap_{\mu_2,\theta}$ for the corresponding branching capacities (with the same jump distribution $\theta$). Then, there is a $C=C(\mu_1,\mu_2)>0$ (which is even independent of $\theta$) such that for every finite subset $A\subseteq\Z^d$,
\begin{equation}\label{MT5-e}
C^{-1}\cdot\BCap _{\mu_1,\theta}(A)\leq\BCap_{\mu_2,\theta}(A)\leq C\cdot\BCap_{\mu_1,\theta}(A).
\end{equation}
\end{theorem}

We believe the following is also true but cannot show it:
\begin{equation}
C^{-1}\cdot\BCap_{\mu,\theta_1}(A)\leq\BCap_{\mu,\theta_2}(A)\leq C\cdot\BCap_{\mu,\theta_1}(A),\quad
\text{for any finite }A\subset \Z^d,
\end{equation}
where $C=C(\mu,\theta_1,\theta_2)$ is independent of $A$.

In \cite{Z162}, we introduce branching recurrence and branching transience and prove an analogous version of Wiener's Test. Hence from Theorem \ref{MT5} one can see that whether a set is branching recurrent or branching transient is somehow independent of the choice of offspring distribution. For more details, see \cite{Z162}.

\section{Preliminaries}
We begin with some notations. For a set $K\subseteq\Z^d$, we write $|K|$ for its cardinality. We write $K\ssubset\Z^d$ to express that $K$ is a finite nonempty subset of $\Z^d$. For $x\in\Z^d$ (or $\R^d$), we denote by $|x|$ the Euclidean norm of $x$. We will mainly use the norm $\|\cdot\|$ corresponding the jump distribution $\theta$, i.e. $\|x\|=\sqrt{x\cdot Q^{-1}x}/\sqrt{d}$, where $Q$ is the covariance matrix of $\theta$. For convenience, we set $|0|=\|0\|=0.5$. We denote by $\diam(K)=\sup\{\| a-b\|:a, b\in K\}$, the diameter of $K$ and by $\Rad(K)=\sup\{\| a\| : a\in K\}$, the radius of $K$ respect to $0$. We write $\Ball(r)$ for the ball $\{z\in\Z^d: \| z\|\leq r\}$ and $\BallE(r)$ for the Euclidean ball $\{z\in\Z^d: |z|\leq r\}$. For any subsets $A,B$ of $\Z^d$, we denote by $\dist(A,B)=\inf\{ \| x-y \|: x\in A, y\in B\}$ the distance between $A$ and $B$. When $A=\{x\}$ consists of just one point, we just write $\dist(x,B)$ instead.
For any path $\gamma:\{0,\dots,k\}\rightarrow \Z^d$, we let $|\gamma|$ stand for $k$, the length, i.e. the number of edges of $\gamma$ , $\widehat{\gamma}$ for $\gamma(k)$, the endpoint of $\gamma$ and $[\gamma]$ for $k+1$, the number of vertices of $\gamma$. Sometimes we just use a sequence of vertices to express a path. For example, we may write $(\gamma(0),\gamma(1),\dots, \gamma(k))$ for the path $\gamma$. For any $B\subset \Z^d$, we write $\gamma\subseteq B$ to express that all vertices of $\gamma$ except the starting point and the endpoint, lie inside $B$, i.e. $\gamma(i)\in B$ for any $1\leq i\leq |\gamma|-1$. If the endpoint of a path $\gamma_1:\{0,\dots,|\gamma_1|\}\rightarrow \Z^d$ coincides with the starting point of another path $\gamma_2:\{0,\dots,|\gamma_2|\}\rightarrow \Z^d$, then we can define the composite of $\gamma_1$ and $\gamma_2$ by concatenating $\gamma_1$ and $\gamma_2$:
$$
\gamma_1\circ\gamma_2:\{0,\dots,|\gamma_1|+|\gamma_2|\}\rightarrow \Z^d,
$$
$$
\gamma_1\circ\gamma_2(i)=\left\{
\begin{array}{ll}
\gamma_1(i),&\text{ for } i\leq|\gamma_1|;\\
\gamma_2(i-|\gamma_1|),&\text{ for } i\geq|\gamma_1|.\\
\end{array}
\right.
$$

We now state our convention regarding constants. Throughout the text (unless otherwise specified), we use $C$ and $c$ to denote positive constants depending only on dimension $d$, the critical distribution $\mu$ and the jump distribution $\theta$, which may change from place to place. Dependence of constants on additional parameters will be made or stated explicit. For example, $C(\lambda)$ stands for a positive constant depending on $d,\mu,\theta, \lambda$. For functions $f(x)$ and $g(x)$, we write $f\sim g$ if $\lim_{x\rightarrow \infty}(f(x)/g(x))=1$. We write $f\preceq g$, respectively $f\succeq g$, if there exist constants $C$ such that, $f\leq Cg$, respectively $f\geq Cg$. We use $f\asymp g$ to express that $f\preceq g$ and $f\succeq g$. We write $f\ll g$ if $\lim_{x\rightarrow \infty}(f(x)/g(x))=0$.

\subsection{Finite and infinite trees.} We are interested in rooted ordered trees (plane trees), in particular, Galton-Watson (GW) trees and its companions. Recall that $\mu=(\mu(i))_{i\in\N}$ is a given critical distribution with finite variance $\sigma^2>0$. Note that we exclude the trivial case that $\mu(1)=1$. Throughout this paper, $\mu$ will be fixed. Define another probability measure $\widetilde{\mu}$ on $\N$, call the \textbf{adjoint measure} of $\mu$ by setting $\widetilde{\mu}(i)=\sum_{j=i+1}^\infty \mu(j)$. Since $\mu$ has mean $1$, $\widetilde{\mu}$ is indeed a probability measure. The mean of $\widetilde{\mu}$ is $\sigma^2/2$. A Galton-Watson process with distribution $\mu$ is a process starting with one initial particle, with each particle having independently a random number of children due to $\mu$. The Galton-Watson tree is just the family tree of the Galton-Watson process, rooted at the initial particle. We simply write \text{$\mu$-GW tree} for the Galton-Watson tree with offspring distribution $\mu$.  If we just change the law of the number of children for the root, using $\widetilde{\mu}$ instead of $\mu$ (for other particles still use $\mu$), the new tree is called an \textbf{adjoint $\mu$-GW tree}. The \textbf{infinite $\mu$-GW tree} is constructed in the following way: start with a semi-infinite line of vertices, called the spine, and graft to the left of each vertex in the spine an independent adjoint $\mu$-GW tree, called a bush. The infinite $\mu$-GW tree is rooted at the first vertex of the spine. Here the left means that we assume every vertex in spine except the root is the youngest child (the latest in the Depth-First search order) of its parent. The \textbf{invariant $\mu$-GW tree} is defined similarly to the infinite $\mu$-GW tree except that we graft to the root (the first vertex of the spine) an independent $\mu$-GW tree instead of an independent adjoint $\mu$-GW tree. The reason for the introduction of these companion trees will be clear in Section 5. Each tree is ordered using the classical order according to Depth-First search starting from the root.

\begin{remark}
The infinite $\mu$-GW tree is somehow one half of the GW tree conditioned on survival in the following sense. The GW tree conditioned on survival has a unique infinite simple path starting from the root. It turns out that the subtree generated by the vertices in the spine and all those vertices on the left of the spine is equivalent to our infinite $\mu$-GW tree. For more details about the GW tree conditioned on survival, one can see e.g. \cite{A91}. The invariant $\mu$-GW tree appears in \cite{LL141}, and admits an invariant shift.
\end{remark}

\subsection{Tree-indexed random walk.} Now we introduce the random walk in $\Z^d$ with jump distribution $\theta$, indexed by a random plane tree $T$. First choose some $a\in\Z^d$ as the starting point. Conditionally on $T$ we assign independently to each edge of $T$ a random variable in $\Z^d$ according to $\theta$. Then we can uniquely define a function $\mathcal{S}_T: T\rightarrow \Z^d$, such that, for every vertex $v \in T$ (we also use $T$ for the set of all vertices of the tree $T$), $\mathcal{S}_T(v)-a$ is the sum of the variables of all edges belonging to the unique simple path from the root $o$ to the vertex $u$ (hence $\mathcal{S}_T(o)=a$). A plane tree $T$ together with this random function $\mathcal{S}_T$ is called $T$-indexed random walk starting from $a$. When $T$ is a $\mu$-GW tree, an adjoint $\mu$-GW tree, and an infinite $\mu$-GW tree respectively, we simply call the tree-indexed random walk a \textbf{snake}, an \textbf{adjoint snake} and an \textbf{infinite snake} respectively. We write $\Snake_x$, $\Snake'_x$ and $\Snake^{\infty}_x$ for a snake, an adjoint snake, and an infinite snake, respectively, starting from $x\in \Z^d$. Note that a snake is just the branching random walk with offspring distribution $\mu$ and jump distribution $\theta$. We also need to introduce the \textbf{reversed infinite snake} starting from $x$, $\Snake^{-}_x$, which is constructed in the same way as $\Snake^{\infty}_x$ except that the variables assigned to the edges in the spine are now due to not $\theta$ but the reverse distribution $\Rtheta$ of $\theta$ (i.e. $\Rtheta(x)\DeFine\theta(-x)$ for $x\in \Z^d$) and similarly the \textbf{invariant snake} starting from $x$, $\Snake^{I}_x$, which is constructed by using the invariant $\mu$-GW tree as the random tree $T$ and using $\Rtheta$ for all edges of the spine of $T$ and $\theta$ for all other edges. For an infinite snake (or reversed infinite snake, invariant snake), the random walk indexed by its spine, called its backbone, is just a random walk with jump distribution $\theta$ (or $\Rtheta$).  Note that all snakes here certainly depend on $\mu$ and $\theta$. Since $\mu$ and $\theta$ are fixed throughout this work, we omit their dependence in the notation.


\subsection{Random walk with killing.} We will use the tools of random walk with killing. Suppose that when the random walk is currently at position $x\in\Z^d$, then it is killed, i.e. jumps to a 'cemetery' state $\delta$, with probability $\KRW(x)$, where $\KRW:\Z^d\rightarrow [0,1]$ is a given function. In other words, the random walk with killing rate $\KRW(x)$ (and jump distribution $\theta$) is a Markov chain $\{X_n:n\geq0\}$ on $\Z^d\cup\{\delta\}$ with transition probabilities $p(\cdot,\cdot)$ given by: for $x,y\in\Z^d$,
\begin{equation*}
p(x,\delta)=\KRW(x), \quad p(\delta, \delta)=1, \quad p(x,y)=(1-\KRW(x))\theta(y-x).
\end{equation*}
For any path $\gamma:\{0,\dots, n\}\rightarrow \Z^d$ with length $n$, its probability weight $\BRW(\gamma)$ is defined to be the probability that the path consisting of the first $n$ steps for the random walk with killing starting from $\gamma(0)$ is $\gamma$. Equivalently,
\begin{equation}\label{def-b}
\BRW(\gamma)=\prod_{i=0}^{|\gamma|-1}(1-\KRW(\gamma(i)))\theta(\gamma(i+1)-\gamma(i))
=\SRW(\gamma)\prod_{i=0}^{|\gamma|-1}(1-\KRW(\gamma(i))),
\end{equation}
where $\SRW(\gamma)=\prod_{i=0}^{|\gamma|-1}\theta(\gamma(i+1)-\gamma(i))$ is the probability weight of $\gamma$ corresponding to the random walk with jump distribution $\theta$. Note that $\BRW(\gamma)$ depends on the killing. We delete this dependence on the notation for simplicity.

Now we can define the corresponding Green function for $x,y\in \Z^d$:
\begin{equation*}
G_\KRW(x,y)=\sum_{n=0}^{\infty}P(S^{\KRW}_x(n)=y)=\sum_{\gamma:x\rightarrow y}\BRW(\gamma).
\end{equation*}
where $S^{\KRW}_x=(S^{\KRW}_x(n))_{n\in \N}$ is the random walk (with jump distribution $\theta$) starting from $x$, with killing function $\KRW$, and the last sum is over all paths from $x$ to $y$. For $x\in \Z^d, A\subset \Z^d$, we write $G_\KRW(x,A)$ for $\sum_{y\in A}G_\KRW(x,y)$.

For any $B\subset \Z^d$ and $x, y\in \Z^d$, define the harmonic measure:
$$
\Hm^{B}_\KRW(x,y)= \sum_{\gamma:x\rightarrow y, \gamma\subseteq B}\BRW(\gamma).
$$

Note that when the killing function $\KRW\equiv 0$, the random walk with this killing is just random walk without killing and we write $\Hm^{B}(x,y)$ for this case.

We will repeatedly use the following First-Visit Lemma. The idea is to decompose a path according to the first or last visit of a set.
\begin{lemma} \label{bd-1}
For any $B\subseteq \Z^d$ and $a\in B, b\notin B$, we have:
$$
G_\KRW(a,b)=\sum_{z\in B^c}\Hm^{B}_\KRW(a,z)G_\KRW(z,b)=\sum_{z\in B}G_\KRW(a,z)\Hm^{B^c}_\KRW(z,b);
$$
$$
G_\KRW(b,a)=\sum_{z\in B}\Hm^{B^c}_\KRW(b,z)G_\KRW(z,a)=\sum_{z\in  B^c}G_\KRW(b,z)\Hm^{B}_\KRW(z,a).
$$
\end{lemma}

\subsection{Some facts about random walk and the Green function.}
For $x\in\Z^d$, we write $S_x=(S_x(n))_{n\in\N}$ for the random walk with jump distribution $\theta$ starting from $S_x(0)=x$. The norm $\| \cdot \|$ corresponding to $\theta$ for every $x\in \Z^d$ is defined to be $\|x\|=\sqrt{x\cdot Q^{-1}x}/\sqrt{d}$, where $Q$ is the covariance matrix of $\theta$. Note that $\|x\|\asymp |x|$, especially, there exists $c>1$, such that $\Ball(c^{-1}n)\subseteq\BallE(n)\subseteq\Ball(cn)$, for any $n\geq1$. The Green function $g(x,y)$ is defined to be:
$$
g(x,y)=\sum_{n=0}^{\infty}P(S_x(n)=y)=\sum_{\gamma:x\rightarrow y}\SRW(\gamma).
$$
We write $g(x)$ for $g(0,x)$.

Our assumptions about the jump distribution $\theta$ guarantee the standard estimate for the Green function (see e.g. Theorem 2 in \cite{U98}):
\begin{equation}\label{green}
g(x)\sim a_d\| x \|^{2-d};
\end{equation}
and (e.g. one can verify this using the error estimate of Local Central Limit Theorem in \cite{U98})
\begin{equation}\label{ngreen}
\sum_{n=0}^{\infty}(n+1)\cdot P(S_0(n)=x)=\sum_{\gamma:0\rightarrow x}[\gamma]\cdot\SRW(\gamma)\asymp \| x\| ^{4-d}\asymp |x|^{4-d}.
\end{equation}
where $a_d=\frac{\mathbf{\Gamma}((d-2)/2))}{2d^{(d-2)/2}\pi^{d/2}\sqrt{\text{det} Q}}$.

Also by LCLT, one can get the following lemma.
\begin{lemma}\label{pro-g2}
\begin{equation}
\lim_{n\rightarrow\infty}sup_{x\in\Z^d}\left(\sum_{\gamma:0\rightarrow x,|\gamma|\geq n|x|^2}
\SRW(\gamma)/g(x)\right)=0.
\end{equation}
\end{lemma}

The following lemma is natural from the perspective of Brownian motion, the scaling limit of random walk.
\begin{lemma}\label{pro-g1}
Let $U,V$ be two connected bounded open subset of $\R^d$ such that $\overline{U}\subseteq V$. Then there exists a $C=C(U,V)$ such that if $A_n=nU\cap\Z^d, B_n=nV\cap\Z^d$ then when $n$ is sufficiently large,
\begin{equation}\label{pro-b1}
\sum_{\gamma:x\rightarrow y, \gamma\subseteq B_n, |\gamma| \leq2n^2}\SRW(\gamma)\geq Cg(x,y), \text{ for any }x,y\in A_n
\end{equation}
\end{lemma}
Since this Lemma may not be standard, we give a sketch of proof here:
\begin{proof}[Sketch of Proof.]
Without loss of generality, one can assume $\theta$ is aperiodic. The first step is to show:
\begin{itemize}
\item
There is a $\delta\in(0,0.1)$, such that, for any $\epsilon>0$ small enough, and $m\in\N^+$ large enough (depending on $\epsilon$), we can find $c_1=c_1(\epsilon)$, such that, for any $n\in [\epsilon m^2,2\epsilon m^2]$, $z,w\in\Ball(3\delta m)$, we have:
\begin{equation}\label{ll1}
p_n^{m}(z,w)\doteq \sum_{\gamma:z\rightarrow w,\gamma\subseteq \Ball (m),|\gamma|=n}\SRW(\gamma)\geq c_1\cdot m^{-d}.
\end{equation}
\end{itemize}
Indeed, the Markov property implies that:
$$
p_n^{m}(z,w)\geq P(S_z(n)=w)-\max\{P(S_y(k)=w):k\leq n, y\in (\Ball (m))^c\},
$$
and the LCLT establishes \eqref{ll1}.
Using this estimate, one can see that:
\begin{itemize}
\item
For any $\epsilon>0$ small enough, and $m\in\N^+$ large enough, we can find $c_2=c_2(\epsilon)$, such that, for any $z,w\in\Ball(3\delta m)$, we have (we write $\Ball_x(r)$ for the ball centered at $x$ with radius $r$):
\begin{equation}\label{ll2}
\sum_{\gamma:z\rightarrow w,|\gamma|\leq 2\epsilon m^2,\gamma\subseteq \Ball(m)}\SRW(\gamma)\geq c_2 m^{2-d};
\end{equation}
\begin{equation*}
\sum_{\gamma:z\rightarrow \Ball_w(\delta m/10),|\gamma|\leq 2\epsilon m^2,\gamma\subseteq \Ball(m)}\SRW(\gamma)\geq c_2 m^2.
\end{equation*}
\end{itemize}
Note that in the first assertion, the left hand side is increasing for $m$ when $z,w$ are fixed. Due to this fact, one can get that
\begin{itemize}
\item
For any $\epsilon>0$ small enough, and $m\in\N^+$ large enough, we can find $c_2=c_2(\epsilon)$, such that, for any $z,w\in\Ball(3\delta m)$, we have:
\begin{equation}\label{ll3}
\sum_{\gamma:z\rightarrow w,|\gamma|\leq 2\epsilon m^2,\gamma\subseteq \Ball(m)}\SRW(\gamma)\geq c_3 \|z-w\|^{2-d};
\end{equation}
\end{itemize}

By considering the first visit of $\Ball_w(\delta m/10)$, one can get:
\begin{align*}
\,&\sum_{\gamma:z\rightarrow \Ball_w(\delta m/10),|\gamma|\leq 2\epsilon m^2,\gamma\subseteq \Ball(m)\setminus \Ball_w(\delta m/10) }\SRW(\gamma)\\
\geq& \sum_{\gamma:z\rightarrow \Ball_w(\delta m/10),|\gamma|\leq 2\epsilon m^2,\gamma\subseteq \Ball(m)}\SRW(\gamma)/\max\{g(x,\Ball(\delta m/10)): x\in\Ball(\delta m/10)\}\\
\asymp& m^2/m^2\asymp1.
\end{align*}
Hence we have:
\begin{itemize}
\item
For any $\epsilon>0$ small enough, and $m\in\N^+$ large enough, we can find $c_4=c_4(\epsilon)$, such that, for any $z,w\in\Ball(3\delta m)$, we have:
\begin{equation}\label{ll4}
\sum_{\gamma:z\rightarrow \Ball_w(\delta m/10),|\gamma|\leq 2\epsilon m^2,\gamma\subseteq \Ball(m)\setminus \Ball_w(\delta m/10) }\SRW(\gamma)\geq c_4.
\end{equation}
\end{itemize}

Now one can show the lemma. Without loss of generality, assume $\dist (U,V^c)=1$. First, choose a finite number of balls with radius $\delta$ and centers at $U$: $B_1,B_2,\dots,B_k$ covering $\overline{U}$. Choose $\epsilon$ small enough for \eqref{ll2},\eqref{ll3},\eqref{ll4} and $\epsilon<1/k$. Now we argue that when $n$ is sufficiently large, \eqref{pro-b1} holds.

Write $B'_i=nB_i\cap \Z^d$ and $\overline{B}'_i=n\overline{B}_i\cap \Z^d$ for $i=1,\dots,k$, where $\overline{B}_i$ is the ball with radius $1$ and the same center of $B_i$. When $\|x-y\|\leq 2\delta n$, by \eqref{ll3} we have \eqref{pro-b1}. Otherwise, $x,y$ are not on the same $B'_i$. However, we can find at most $k+1$ points $x_0=x,x_1,\dots,x_l=y$,($l\leq k$) such that $x_j$ and $x_{j+1}$ are in the same $B'_i$, say $B'_j$. Note that when $z,w$ are on the same $B'_i$, by \eqref{ll4}, for any $z'\in \Ball_z(\delta n/10)$,
$$
\sum_{\gamma:z'\rightarrow \Ball_w(\delta n/10),|\gamma|\leq 2\epsilon n^2,\gamma\subseteq \overline{B}'_i\setminus \Ball_w(\delta n/10) }\SRW(\gamma)\geq c_4.
$$
Hence, by connecting paths, one can get:
\begin{align*}
&\sum_{\gamma:x\rightarrow y, \gamma\subseteq B_n, |\gamma| \leq2n^2}\SRW(\gamma)\geq \sum_{\gamma_0:x_0\rightarrow \Ball_{x_1}(\delta n/10),|\gamma_0|\leq 2\epsilon n^2,\gamma_0\subseteq \overline{B}'_0\setminus \Ball_{x_1}(\delta n/10) }\SRW(\gamma_0)\\
&\cdot\sum_1\SRW(\gamma_1)\cdot\sum_2\SRW(\gamma_2)\cdot\dots\sum_{l-2}\SRW(\gamma_{l-2})
\cdot \sum_{\gamma_{l-1}:\widehat{\gamma}_{l-2}\rightarrow y,|\gamma_{l-1}|\leq 2\epsilon n^2,
\gamma_{l-1}\subseteq \overline{B}'_{l-1}}\SRW(\gamma_{l-1})\\
&\geq (c_4)^{l-1}\cdot c_2(n^{2-d})\geq (c_4)^k c_2 n^{2-d}\asymp g(x,y),\\
\end{align*}
where $\sum_j=\sum_{\gamma_j:\widehat{\gamma}_{j-1}\rightarrow \Ball_{x_j}(\delta n/10),|\gamma_j|\leq 2\epsilon n^2,\gamma_j\subseteq \overline{B}'_j\setminus \Ball_{x_{j+1}}(\delta n/10)}$ for $j=1,\dots, l-2$.

\end{proof}

Since our jump distribution $\theta$ may be unbounded, we need the following Overshoot Lemma:
\begin{lemma}\label{overshoot}
For any $r,s>1$, let $B=\Ball(r)$. Then for any $a\in B$, we have:
\begin{equation}\label{os}
\sum_{y\in(\Ball (r+s))^c} \Hm_\KRW^{B}(a,y)\preceq\frac{r^2}{s^{d}},\quad
\sum_{y\in(\Ball (r+s))^c} \Hm_\KRW^{B}(y,a)\preceq\frac{r^2}{s^{d}}.
\end{equation}
\end{lemma}
\begin{proof}
It suffices to show the case when $\KRW\equiv0$. By considering where the last position is before leaving $\Ball(r)$, one can get:
\begin{multline*}
\sum_{y\in(\Ball (r+s))^c} \Hm^{B}(a,y)
\leq\sum_{z\in \Ball(r)}g(a,z)P(\text{the jump leaving }\Ball(r)\geq s)\\
\stackrel{\eqref{as1}}{\leq}(\sum_{z\in \Ball(r)}g(a,z))\cdot C/ s^{d}
\preceq\frac{r^2}{s^{d}}.\\
\end{multline*}
One can show the other inequality similarly.
\end{proof}

\section{Escape probabilities and branching capacity.}
For any $K\ssubset \Z^d$, we are interested in the probability of visiting $K$ by the critical branching random walk with offspring distribution $\mu$ and jump distribution $\theta$, or equivalently, a snake. For any $x\in \Z^d$, write $\pS(x)$, $\rS(x)$, $\qS(x)$ and $\qrS(x)$, respectively, for the probability that a snake, an adjoint snake, an infinite snake and a reversed infinite snake, respectively, starting from $x$ visits $K$, i.e. $P((\Snake_T(T)\cap K)\neq \emptyset)$ where $T,\Snake_T$ are the corresponding random tree and random map. We write $\PS(x)$ and $\RS(x)$ respectively for the probability that a snake and an adjoint snake respectively, starting from $x$ visits $K$ strictly after time zero, i.e. $P((\Snake_T(T\setminus \{o\})\cap K)\neq \emptyset)$. Note that when $x\notin K$, $\pS(x)=\PS(x)$ and $\rS(x)=\RS(x)$. We delete the dependence on $K$ in the notations since we will fix $K$ until Section 9.

We first give some preliminary upper bounds for the visiting probabilities by computing the expectation of the number of visits. Here are the computations. When $x$ is relatively far from $K$, say $\dist(x,K)\geq 2\diam(K)$. For the snake $\Snake_x$, the expectation of the number of offspring at $n$-th generation is one. Hence, the expectation of the number of visiting any $a\in K$ is just $g(x,a)\asymp \|x-a\|^{2-d}\asymp \|x\|^{2-d}$. For the adjoint snake $\Snake'_x$, the expectation of the number of offspring at $n$-th generation (for $n\geq 1$) is $\E\widetilde{\mu}=\sigma^2/2\asymp1$ (recall that $\mu$ is fixed). Hence the expectation of the total number of visiting $a$ can also be bounded by $g(x,a)$ up to some constant multiplier. For the infinite snake $\Snake^\infty_x$, one can see that the expectation of the number of offspring at $n$-th generation is $1+n\cdot\E\widetilde{\mu}\asymp n+1$. Hence when $\dist(x,K)\geq 2\diam(K)$, the expectation of the total number of visiting $a$ is bounded, up to some constant, by:
$$
\sum_{n=0}^{\infty} (n+1)P(S_x(n)=a)\stackrel{\eqref{ngreen}}{\asymp}\|x-a\|^{4-d}\asymp\|x\|^{4-d}.
$$
Recall that $S_x=(S_x(n))_{n\in\N}$ is the random walk starting from $x$ with jump distribution $\theta$. Summing up over all $a\in K$, we get
\begin{equation}\label{1bd}
\begin{array}{l}
\pS(x)\preceq |K|/\|x\|^{d-2};\\
\rS(x)\preceq |K|/\|x\|^{d-2};\\
\qS(x)\preceq |K|/\|x\|^{d-4}.\\
\end{array}
\end{equation}
For $\qrS(x)$, one can also see that the expectation of the number of visiting points is:
$$
\sum_{y\in\Z^d}g^-(x,y)g(y,K)\asymp|K|\sum_{y\in \Z^d}\|x-y\|^{2-d}\|y\|^{2-d}\asymp |K|/\|x\|^{d-4},
$$
where $g^-(x,y)=g(y,x)$ is the Green function for the reversed random walk.

From this, we see that when $x$ tends to infinity, all four types of visiting probabilities tend to $0$. Now we introduce the escape probabilities.
\begin{definition}\label{BCap}
$K$ is a finite subset of $\Z^d$, for any $x\in \Z^d$, define $\EsC_K(x)$ to be the probability that a reversed infinite snake starting from $x$ does not visit $K$ except possibly for the image of the bush grafted to the root and $\Esc _K(x)$ to be the probability that an invariant snake starting from $x$ does not visit $K$ except possibly for the image of the spine. Define the \textbf{Branching capacity} of $K$ by:
\begin{equation}\label{def_BCap}
\BCap (K)=\sum_{a\in K}\EsC _K(a)=\sum_{a\in K}\Esc _K(a).
\end{equation}
\end{definition}

\begin{remark}
\cite{ARZ15} constructs the model of branching interlacement. As a main step, they gives the definition of branching capacity (only) when $\mu$ is the critical geometric distribution. In that case, the branching capacity here is equivalent to the branching capacity there, up to a constant factor 2. But here we do not need the so-called contour function which plays an important role there. Furthermore, we can construct the model of branching interlacement for general critical offspring distribution. For more details about this, see the forthcoming paper \cite{Z164}.
\end{remark}

The last equality can be seen from our main theorem of branching capacity, Theorem \ref{MT1}. We also introduce the escape probability for the infinite snake $\EsC^+_K(x)$, which is defined to be the probability that an infinite snake starting from $x$ does not visit $K$ except possibly for the image of the bush grafted to the root. Note that $\EsC^+_K(x)\leq 1-\qS(x)\rightarrow 1$, as $n\rightarrow \infty$.

\begin{remark}\label{compare}
If we let $\mu$ be the degenerate measure, that is, $\mu(1)=1$, then: the snake and the infinite snake are just the random walk with jump distribution $\theta$; the reversed infinite snake and the invariant snake are the random walk with jump distribution $\Rtheta$. Then $\EsC_K$ is just the escape probability for the 'reversed' walk and $\Esc_K$ is the escape probability for the 'original' walk. In that case, Theorem \ref{MT1} is just the classical theorem for regular capacity.
Note that when $\theta$ is symmetric, for random walk, $\EsC_K(a)=\Esc_K(a)$. But this is generally not true for branching random walk even when $\theta$ is symmetric. If $K=\{a\}$ consists of only one point, then it is true by Theorem \ref{MT1}
\end{remark}

\section{Monotonicity and subadditivity.} We postpone the proof of Theorem \ref{MT1} until Section 7. We now state some basic properties about branching capacity. Like regular capacity, branching capacity is monotone and subadditive:
\begin{prop}\label{mono}
For any $K\subset K'$ finite subsets of $\Z^d$,
\begin{equation*}
\BCap(K)\leq \BCap(K');
\end{equation*}
For any $K_1,K_2$ finite subsets of $\Z^d$,
\begin{equation*}
\BCap(K_1\cap K_2)+\BCap(K_1\cup K_2)\leq \BCap(K_1)+\BCap(K_2).
\end{equation*}
\end{prop}

\begin{proof}
When $K\subset K'$, a snake visiting $K$ must visit $K'$. So
$$
P(\Snake_x \text{ visits } K)\leq P(\Snake_x \text{ visits } K').
$$
By \eqref{m1-1}, we get $\BCap(K)\leq \BCap(K')$.
For the other inequality, we use a similar idea. First, we have:
$$
P(\Snake_x \text{ visits } K_1)= P(\Snake_x \text{ visits } K_1 \text{ but not }K_2)+P(\Snake_x \text{ visits both } K_1 \& K_2);
$$
$$
P(\Snake_x \text{ visits } K_2)= P(\Snake_x \text{ visits } K_2 \text{ but not }K_1)+P(\Snake_x \text{ visits both } K_1 \& K_2);
$$
\begin{multline*}
P(\Snake_x \text{ visits } K_1\cup K_2)=P(\Snake_x \text{ visits } K_1 \text{ but not }K_2)+\\
P(\Snake_x \text{ visits } K_2 \text{ but not }K_1)+P(\Snake_x \text{ visits both } K_1 \& K_2).
\end{multline*}
Since $P(\Snake_x \text{ visits } K_1\cap K_2)\leq P(\Snake_x \text{ visits both } K_1 \& K_2)$, we have:
\begin{multline*}
P(\Snake_x \text{ visits } K_1\cup K_2)+P(\Snake_x \text{ visits } K_1\cap K_2)
\leq P(\Snake_x \text{ visits } K_1)+P(\Snake_x \text{ visits } K_2).
\end{multline*}
This concludes the proposition by \eqref{m1-1}.
\end{proof}

\section{Random walk with special killing.} We begin with some straightforward computations. When a snake $\Snake_x=(T,\mathcal{S}_T)$ visits $K$, since $T$ is an ordered tree, we have the unique first vertex, denoted by $\tau_K$, in $\{v\in T: \mathcal{S}_T(v)\in K\}$ due to the default order. We say $\mathcal{S}_T(\tau_K)$ is the visiting point or $\Snake_x$ visits $K$ at $\mathcal{S}_T(\tau_K)$. Assume $(v_0,v_1,\dots,v_k)$ is the unique simple path in $T$ from the root $o$ to $\tau_K$. Define $\bb(\Snake_x)=(\mathcal{S}_T(v_0),\mathcal{S}_T(v_1),\dots,\mathcal{S}_T(v_k))$ and say $\Snake_x$ visits $K$ via $\bb(\Snake_x)$. We now compute $P(\bb(\Snake_x)=\gamma)$, for any given $\gamma=(\gamma(0),\dots,\gamma(k))\subseteq K^c$ starting from $x$, ending at $K$. Let $\widetilde{a}_i$ and $\widetilde{b}_i$ respectively, be the number of the older, and younger respectively, brothers of $v_i$, for $i=1,\dots,k$. From the tree structure, one can see that, for any $l_1,\dots,l_k,$ $m_1,\dots,m_k\in\N$,
\begin{multline}\label{key1}
P(\Snake_x \text{ visits }K \text{ via }\gamma; \widetilde{a}_i=l_i, \widetilde{b}_i=m_i, \text{for }i=1,\dots,k)\\
=\SRW(\gamma)\prod_{i=1}^{k}\left(\mu(l_i+m_i+1)(\rp(\gamma(i-1)))^{l_i}\right),
\end{multline}
where $\rp(z)$ is the probability that a snake starting from $z$ does not visit $K$ conditioned on the initial particle having only one child. Summing up, we get:
\begin{align*}
&P(\Snake_x \text{ visits }K \text{ via }\gamma)\\
=&\sum_{l_1,\dots, l_k;m_1,\dots,m_k\in\N}P(\Snake_x \text{ visits }K \text{ via }\gamma; \widetilde{a}_i=l_i, \widetilde{b}_i=m_i, \text{for }i=1,\dots,k)\\
=&\sum_{l_1,\dots, l_k;m_1,\dots,m_k\in\N}\SRW(\gamma)\prod_{i=1}^{k}\left(\mu(l_i+m_i+1)(\rp(\gamma(i-1)))^{l_i}\right)\\
=&\SRW(\gamma)\prod_{i=1}^{k}\sum_{l_i,m_i\in\N}\left(\mu(l_i+m_i+1)(\rp(\gamma(i-1)))^{l_i}\right)\\
=&\SRW(\gamma)\prod_{i=1}^{k}\sum_{l_i\in\N}\left(\widetilde{\mu}(l_i)(\rp(\gamma(i-1)))^{l_i}\right).\\
\end{align*}
Note that for any $z\notin K$,
\begin{equation*}
\sum_{l\in\N}\widetilde{\mu}(l)(\rp(z))^{l}
\end{equation*}
is just $1-\rS(z)$, the probability that an adjoint snake starting form $z$ does not visit $K$. If we let the killing function be
\begin{equation}\label{killing}
\KRW(x)=P(\Snake'_x\text{ visits } K )=
\rS(x).
\end{equation}
then we have (recall the definition of $\BRW(\gamma)$ from \eqref{def-b})
\begin{align*}
\BRW(\gamma)=&\SRW(\gamma)\prod_{i=1}^{k}\left(1-\KRW(\gamma(i-1))\right)
=\SRW(\gamma)\prod_{i=1}^{k}\left(1-\rS(\gamma(i-1))\right)\\
=&\SRW(\gamma)\prod_{i=1}^{k}\sum_{l_i\in\N}\left(\widetilde{\mu}(l_i)(\rp(\gamma(i-1)))^{l_i}\right)=P(\Snake_x \text{ visits }K\text{ via }\gamma).
\end{align*}
This brings us to the key formula of this work:
\begin{prop}\label{key-1}
\begin{equation}\label{key}
\BRW(\gamma)=P(\Snake_x \text{ visits }K\text{ via }\gamma).
\end{equation}
\end{prop}
In words, the probability that a snake visits $K$ via $\gamma$ is just $\gamma$'s probability weight according to the random walk with the killing function given by \eqref{killing}. Throughout this work, we will mainly use this killing function and write $G_K(\cdot, \cdot)$ for the corresponding Green function. By summing the last equality over $\gamma$, we get: for any $a\in K$,
\begin{equation}\label{p1}
P(\Snake_x\text{ visits } K \text{ at } a)=\sum_{\gamma:x\rightarrow a}\BRW(\gamma)=G_K(x,a);
\end{equation}
and
\begin{equation}\label{p2}
\pS(x)=P(\Snake_x\text{ visits } K)=\sum_{\gamma:x\rightarrow K}\BRW(\gamma)=G_K(x,K).
\end{equation}
Note that since $\rS(x)=1$ for $x\in K$, when $\gamma$, except for the ending point, intersects $K$, $\BRW(\gamma)=0$.

On the other hand, from the structure of the infinite snake, one can easily see that $\qS(x)$ is just the probability that in this killing random walk, a particle starting at $x$ will be killed at some finite time.

Now we turn to the last visiting point, which can be addressed similarly. When a snake $\Snake_x=(T,\mathcal{S}_T)$ visits $K$, denoted by $\xi_K$, the last vertex in $\{v\in T: \mathcal{S}_T(v)\in K\}$ due to the default order. Assume $(v_0,v_1,\dots,v_k)$ is the unique simple path in $T$ from the root $o$ to $\xi_K$. Define $\overline{\bb}(\Snake_x)=(\mathcal{S}_T(v_0),\mathcal{S}_T(v_1),\dots,\mathcal{S}_T(v_k))$ and say $\Snake_x$ leaves $K$ at $\mathcal{S}_T(v_k)$, via $\overline{\bb}(\Snake_x)$. We would like to compute $P(\overline{\bb}(\Snake_x)=\gamma)$, for any $\gamma=(\gamma(0),\dots,\gamma(k))$ starting from $x$ and ending at $A$ (note that unlike the former case, the interior of $\gamma$ now may intersect $K$). Let $\widetilde{a}_i$ ($\widetilde{b}_i$ respectively) be the number of the older (younger respectively) brothers of $v_i$, for $i=1,\dots,k$. Similar to the former case, one can see that, for any $l_1,\dots,l_k,$ $m_1,\dots,m_k\in\N$,
\begin{multline}\label{key2}
P(\Snake_x \text{ leaves }K \text{ via }\gamma; \widetilde{a}_i=l_i, \widetilde{b}_i=m_i, \text{for }i=1,\dots,k)\\
=\SRW(\gamma)(1-\PS(\gamma(k)))\prod_{i=1}^{k}\left(\mu(l_i+m_i+1)(\hat{r}(\gamma(i-1)))^{m_i}\right),
\end{multline}
where $\hat{r}(z)$ is the probability that a snake starting from $z$ does not visit (except possibly for the root) $K$ conditioned on the initial particle having only one child. Summing up, we get:
\begin{equation}\label{last}
P(\Snake_x \text{ leaves }K \text{ via }\gamma)=\SRW(\gamma)(1-\PS(\widehat{\gamma}))\prod_{i=1}^{k}(1-\RS(\gamma(i-1))).
\end{equation}
If we let the killing function be $\KRW'(x)=\RS(x)$, then the last term  is just $(1-\PS(\widehat{\gamma}))\BRW_{\KRW'}(\gamma)$.

\begin{remark}
We will always use the killing function in \eqref{killing}, except in the proof of \eqref{m1-3}.
\end{remark}

\begin{remark}\label{r1}
Now the reason for the introduction of the adjoint snake and the infinite snakes is clear: in order to understand $\pS(x)$, the probability of visiting $K$, we need to study the random walk with killing where the killing function is just the probability of the adjoint snake visiting $K$.
\end{remark}

\begin{remark}
The computations here are initiated in \cite{Z15}. Note that in this section, we do not need the assumption $d\geq5$.
\end{remark}

\section{Convergence of the Green function.} The goal of this section is to prove:
\begin{lemma}\label{Conver_G}
\begin{equation}\label{Conver-G}
\lim_{x,y\rightarrow \infty}G_K(x,y)/g(x,y)=1.
\end{equation}
\end{lemma}

\begin{proof}
The part of '$\leq$' is trivial, since $G_K(x,y)\leq g(x,y)$. We need to consider the other part.

First, consider the case $\|x\|/2\leq\|y\|\leq2\|x\|^{1.1}$. Let
\begin{align*}
&\Gamma_1=\{\gamma:x\rightarrow y| |\gamma|\geq \|x\|^{0.1}\cdot \|x-y\|^2\};\\
&\Gamma_2=\{\gamma:x\rightarrow y| \gamma \text{ visits } \Ball(\|x\|^{0.9})\}.\\
\end{align*}
By Lemma \ref{pro-g2}, one can see that $\sum_{\gamma\in \Gamma_1}\SRW(\gamma)/g(x,y)$ tends to 0. Similar to  the First-Visit Lemma, by considering the first visiting place, we have (let $B=\Ball(\|x\|^{0.9})$):
\begin{align*}
\sum_{\gamma\in \Gamma_2}\SRW(\gamma)=&\sum_{a\in B}\Hm^{B^c}(x,a)
g(a,y)\asymp \sum_{a\in B}\Hm^{B^c}(x,a)\|y\|^{2-d}\\
=&P(S_x \text{ visits } B)\cdot \|y\|^{2-d}\asymp  (\|x\|^{0.9}/\|x\|)^{d-2} \|y\|^{2-d}\\
\preceq & \|x\|^{-0.1} \|x-y\|^{2-d}\asymp  \|x\|^{-0.1}g(x,y).\\
\end{align*}
Note that the estimate of $P(S_x \text{ visits } \Ball(r))\asymp (r/\|x\|)^{d-2}$ is standard, and for the second last inequality we use $\|y\|\geq(\|x\|+\|y\|)/3\succeq\|x-y\|$.
Hence, we get $\sum_{\gamma\in \Gamma_2}\SRW(\gamma)/g(x,y)\rightarrow 0$ and therefore,
\begin{equation}\label{a1}
\sum_{\gamma:x\rightarrow y, \; \gamma\notin \Gamma_1\cup\Gamma_2}\SRW(\gamma)\sim g(x,y).
\end{equation}

For any $\gamma:x\rightarrow y, \; \gamma\notin \Gamma_1\cup\Gamma_2$, using \eqref{1bd}, one can see:
\begin{align*}
\BRW(\gamma)/\SRW(\gamma)=&\prod_{i=0}^{|\gamma|-1}\left(1-\KRW(\gamma(i))\right)\geq(1-c|K|/(\|x\|^{0.9})^{d-2})^{|\gamma|}\\
\geq & 1-c|K||\gamma|/(\|x\|^{0.9})^{d-2}
\geq 1-c|K|\|x\|^{0.1}\|x-y\|^2/(\|x\|^{0.9})^{d-2} \\
\geq &1-c|K|\|x\|^{0.1}\|x\|^{2.2}/\|x\|^{0.9\cdot 3}
\geq  1-c|K|/\|x\|^{0.4}\rightarrow 1.\\
\end{align*}
Hence, we have:
\begin{equation*}
\sum_{\gamma:x\rightarrow y, \; \gamma\notin \Gamma_1\cup\Gamma_2}\BRW(\gamma)
\sim \sum_{\gamma:x\rightarrow y, \; \gamma\notin \Gamma_1\cup\Gamma_2}\SRW(\gamma).
\end{equation*}
Combining this and \eqref{a1}, we get: when $\|x\|/2\leq\|y\|\leq2\|x\|^{1.1}$, \eqref{Conver-G} is true.

When $\|y\|>2\|x\|^{1.1}$, we know $g(x,y)\sim a_d\|y\|^{2-d}$. Hence, we need to show: $G_K(x,y)\sim a_d\|y\|^{2-d}$.
Let $r=2\|y\|^{1/1.1}$ and $B=\Ball(r)$. Then for any $a\in \Ball(2r)\setminus \Ball(r)$, $\|x\|<\|a\|<\|y\|\leq 2\|a\|^{1.1}$ (when $\|y\|$ is large). Hence $G_K(a,y)\sim g(a,y)\sim a_d\|y\|^{2-d}$. Applying the First-Visit Lemma, we have:
\begin{align*}
G_K(x,y)=&\sum_{a\in B^c}\Hm^{B}_\KRW(x,a)G_K(a,y)
\geq\sum_{a\in \Ball(2r)\setminus B}\Hm^{B}_\KRW(x,a)G_K(a,y)\\
\sim &\sum_{a\in\Ball(2r)\setminus B}\Hm^{B}_\KRW(x,a) a_d\|y\|^{2-d}\\
=&(\sum_{a\in B^c}\Hm^{B}_\KRW(x,a)-\sum_{a\in (\Ball(2r))^c}\Hm^{B}_\KRW(x,a))a_d\|y\|^{2-d}\\
\geq & ((1-\rS(x))\EsC^+_K(x)-C\frac{r^2}{r^{d}})a_d\|y\|^{2-d}\\
\sim &a_d\|y\|^{2-d}.\\
\end{align*}
In the second last inequality we use the Overshoot Lemma and
$$
\sum_{a\in B^c}\Hm^{B}_\KRW(x,a)\geq\sum_{a\in B^c}\Hm^{B}_\KRW(x,a)(1-\rS(a))\EsC^+_K(a)=(1-\rS(x))\EsC^+_K(x)\rightarrow 1.
$$
Now, we show \eqref{Conver-G} for the case $\|x\|\leq \|y\|$. The case of $\|x\|\geq \|y\|$ can be handled similarly.
\end{proof}
\begin{remark}
As we have seen in the proof, since the jump distribution $\theta$ maybe unbounded, we need an extra step to control the long jump, via the Overshoot Lemma. This happens again and again later. It might be convenient, especially for a first-time reader, to restrict the attention to the jump distribution with finite range.
\end{remark}

\section{Proof of Theorem \ref{MT1}.}
Now we are ready to prove Theorem \ref{MT1}. It is sufficient to prove:
\begin{lemma}Under the same assumption of Theorem \ref{MT1}, we have:
\begin{equation}
P(\Snake_x \text{ visits } K \text { at }a)\sim a_d\|x\|^{2-d}\EsC_K(a);
\end{equation}
\begin{equation}
P(\Snake_x \text{ leaves } K \text { at }a)\sim a_d\|x\|^{2-d}\Esc_K(a);
\end{equation}
whenever the escape probability on the right hand side is nonzero.
\end{lemma}
\begin{proof}
Fix some $\alpha\in(0,2/(d+2))$. Let $r=\|x\|^{\alpha}, s=\|x\|^{1-\alpha}$ and $B=\Ball(r), B_1=\Ball(s)\setminus B$ and $B_2=(\Ball(s))^c$.
Note that our choice of $\alpha$ implies $r^2/s^{d}\ll \|x\|^{2-d}$. Then,
\begin{multline}\label{eq1}
P(\Snake_x \text { visits }K \text{ at }a)\stackrel{\eqref{p1}}{=}\sum_{\gamma: x\rightarrow a}\BRW(\gamma)=\sum_{b\in B^c} G_K(x,b)\Hm^{B}_\KRW(b,a)\\
=\sum_{b\in B_1} G_K(x,b)\Hm^{B}_\KRW(b,a)+\sum_{b\in B_2} G_K(x,b)\Hm^{B}_\KRW(b,a).
\end{multline}
We argue that the first term has the desired asymptotics and the second is negligible:
\begin{align*}
\sum_{b\in B_1} &G_K(x,b)\Hm^{B}_\KRW(b,a)\stackrel{\eqref{Conver-G}}{\sim} a_d\|x\|^{2-d}\sum_{b\in B_1}\Hm^{B}_\KRW(b,a)\\
&\sim a_d\|x\|^{2-d} (\EsC_K(a)-\sum_{b\in B_2}\Hm^{B}_\KRW(b,a))\\
&\stackrel{\eqref{os}}{\preceq} a_d\|x\|^{2-d}(\EsC_K(a)-r^2/s^{d})\sim a_d\|x\|^{2-d}\EsC_K(a);\\
\end{align*}
\begin{align*}
\sum_{b\in B_2} G_K(x,b)\Hm^{B}_\KRW(b,a)\preceq \sum_{b\in B_2}\Hm^{B}_\KRW(b,a)
\stackrel{\eqref{os}}{\preceq} r^2/s^{d}\ll \|x\|^{2-d}.\\
\end{align*}
This completes the proof of the first assertion.
Very similar arguments can be used for the second assertion. Note that due to \eqref{last}, we need to use the killing function $\KRW'(x)=\RS(x)$ and the analogous version of Lemma \ref{Conver_G} for this killing. We leave the details to the reader.
\end{proof}

\section{The asymptotics for $\qS(x)$, $\qrS(x)$ and $\rS(x)$.}
Thanks to Theorem \ref{MT1}, we also can find the exact asymptotics of the visiting probabilities by an adjoint snake, an infinite snake and a reversed infinite snake, i.e. $\rS(x)$, $\qS(x)$ and $\qrS(x)$:
\begin{prop}\label{q-r}
\begin{equation}\label{r-lim}
\rS(x)\sim \frac{ a_d\sigma^2\BCap(K)}{2\|x\|^{d-2}},
\end{equation}
\begin{equation}\label{q-lim}
\qS(x)\sim \frac{t_d\cdot a_d^2\sigma^2\BCap(K)}{2\|x\|^{d-4}},
\end{equation}
\begin{equation}\label{q_lim}
\qrS(x)\sim \frac{t_d\cdot a_d^2\sigma^2\BCap(K)}{2\|x\|^{d-4}},
\end{equation}
where $\sigma^2$ is the variance of $\mu$, $t_d=t_d(\theta)=\int_{t\in\R^d}\|t\|^{2-d}\|h-t\|^{2-d}dt$, and $h\in \R^d$ is any vector satisfying $\|h\|=1$.
\end{prop}
\begin{remark}
In fact, $t_d=t_d(\theta)$ has the following form:
$$
t_d=\int_{t\in\R^d}\|t\|^{2-d}\|h-t\|^{2-d}dt=d^{d/2}\sqrt{\det Q}\int_{t\in\R^d}|t|^{2-d}|h'-t|^{2-d}dt,
$$
where $h'\in\R^d$ is any vector with $|h'|=1$ in $\R^d$.
\end{remark}
\begin{proof}
Let $\widetilde{s}(x)$ be the probability that a snake starting from $x$ visits $K$ conditioned on the initial particle having exactly one child. Then it is straightforward to see that: when $x\notin K$,
\begin{equation}\label{p-rr}
1-\pS(x)=\sum_{i\in\N}\mu(i)(1-\widetilde{s}(x))^i,\quad
1-\rS(x)=\sum_{i\in\N}\tilde{\mu}(i)(1-\widetilde{s}(x))^i.
\end{equation}
Note that
$$
\sum_{i\in\N}\mu(i)(1-\widetilde{s}(x))^i\geq\sum_{i\in\N}\mu(i)(1-i\widetilde{s}(x))=1-(\E \mu)\widetilde{s}(x)
$$
and
$$
\sum_{i\in\N}\mu(i)(1-\widetilde{s}(x))^i\leq \mu(0)+\sum_{i=1}^{\infty}\mu(i)(1-\widetilde{s}(x))=1-(1-\mu(0))\widetilde{s}(x)
$$
Hence we have
\begin{equation}\label{rr-p}
\pS(x)\asymp \widetilde{s}(x),
\end{equation}
and similarly one can get $\rS(x)\asymp \widetilde{s}(x)$. Therefore,
\begin{equation}\label{r-p}
\rS(x)\asymp \pS(x).
\end{equation}

We will use the following easy lemma and omit its proof.
\begin{lemma}
Let $(a_n)_{n\in \N}$ be any nonnegative sequence satisfying: $\sum_{n\in\N} a_n=1$ and $\sum_{n\in\N} na_n<\infty$. Let $f(t)=\sum_{n\in\N}a_nt^n$.
Then we have:
\begin{equation*}
\lim_{t\rightarrow 1^-}(1-f(t))/(1-t)=\sum_{n\in\N} na_n.
\end{equation*}
\end{lemma}
By this lemma and \eqref{p-rr}, we have
$$
\pS(x)\sim \sum_ii\mu(i)\widetilde{s}(x)=\widetilde{s}(x),
$$
$$
\rS(x)\sim \sum_ii\tilde{\mu}(i)\widetilde{s}(x)=\frac{\sigma^2}{2}\widetilde{s}(x).
$$
Hence,
$$
\rS(x)\sim\frac{\sigma^2}{2}\pS(x)\sim\frac{\sigma^2a_d \BCap(K)}{2\|x\|^{d-2}}.
$$

Now we turn to the asymptotic of $\qS(x)$. We point out two formulas for $\qS(x)$:
\begin{equation}\label{for-q2}
\qS(x)=\sum_{y\in \Z^d}G_K(x,y)\rS(y);
\end{equation}
\begin{equation}\label{for-q1}
\qS(x)=\sum_{y\in \Z^d}g(x,y)\rS(y)\EsC^+_K(y).
\end{equation}
The first can be easily derived by considering where the particle dies in the model of random walk with killing function $\rS$. For the second one, we need to consider a bit different but equivalent model: a particle starting from $x$ executes a random walk, but at each step, the particle has the probability $\rS$ to get a flag (instead of to die) and its movements are unaffected by flags. Let $\tau$ and $\xi$ be the first and last time getting flags (if there is no such times then denote $\tau=\xi=\infty$). Note that since $\qS(z)<1$ (when $|z|$ is large), the total number of flags gained is finite, almost surely. Hence $P(\tau<\infty)=P(\xi<\infty)$ and $\qS(x)$ is just the probability that $\xi<\infty$. By considering where the particle gets its last flag, one can get \eqref{for-q1}.

We will use the following easy lemma and omit its proof:
\begin{lemma}
\begin{equation}\label{t-d}
\|x\|^{d-4}\sum_{z\in \Z^d}\frac{1}{\|z\|^{d-2}\|x-z\|^{d-2}}\sim \int_{t\in\R^d}\|t\|^{2-d}\|h-t\|^{2-d}dt.
\end{equation}
\end{lemma}

For the asymptotics of $\qS(x)$, one can use either \eqref{for-q2} or \eqref{for-q1} and the process is similar to each other. Here we use \eqref{for-q2}. Let $B=\Ball(r)$ and $r$ be very large. Divide the right hand site of \eqref{for-q2} into three parts: $\sum_{z\in B}$, $\sum_{z\in x+B}$ and $\sum_{z\notin B\cup x+B}$. We will argue that the first two parts are negligible compared to $\|x\|^{4-d}$ and the third term has the desired asymptotics.
For the first part, we have:
\begin{multline*}
\|x\|^{d-4}\sum_{y\in B}G_K(x,y)\rS(y)\preceq \|x\|^{d-4}\sum_{y\in B}g(x,y)\cdot 1\\
\preceq\|x\|^{d-4}\sum_{y\in B}\frac{1}{(\|x\|-r)^{d-2}}\preceq\|x\|^{d-4}r^d/(\|x\|-r)^{d-2}\rightarrow 0\;(\text{when }x\rightarrow \infty).\\
\end{multline*}
For the second part, we have:
\begin{multline*}
\|x\|^{d-4}\sum_{y\in x+B}G_K(x,y)\rS(y)\preceq \|x\|^{d-4}\sum_{y\in x+B} 1\cdot \rS(y)\\
\stackrel{\eqref{1bd}}{\preceq}\|x\|^{d-4} \sum_{y\in x+B}|K|\|y\|^{2-d}
\leq\|x\|^{d-4}r^d|K|/(\|x\|-r)^{d-2}\rightarrow 0.\\
\end{multline*}
When $r$ and $\|x\|$ are large and $y\notin B\cup(x+B)$, the ratio between $G_K(x,y)\rS(y)$ and $a_d\|x-y\|^{2-d}a_d\sigma^2\BCap(K)\|y\|^{2-d}/2$ is very close to 1. On the other hand,
\begin{align*}
&\|x\|^{d-4}\sum_{y\notin B\cup(x+B)}a_d\|x-y\|^{2-d}a_d\sigma^2\BCap(K)\|y\|^{2-d}/2\\
=&a_d^2\sigma^2\BCap(K)/2\cdot\|x\|^{d-4}\sum_{y\notin B\cup(x+B)}\|x-y\|^{2-d}\|y\|^{2-d}\\
=&a_d^2\sigma^2\BCap(K)/2 \cdot(\|x\|^{d-4}\sum_{y\in\Z^d}\|x-y\|^{2-d}\|y\|^{2-d}-\\
&\|x\|^{d-4}\sum_{y\in(B\cup x+B)}\|x-y\|^{2-d}\|y\|^{2-d}).\\
\end{align*}
By \eqref{t-d}, the first term in the bracket tends to $t_d$. Similar to the estimate for the first two parts, one can verify that
$$
\|x\|^{d-4}\sum_{y\in(B\cup x+B)}\|x-y\|^{2-d}\|y\|^{2-d}\preceq\|x\|^{d-4}r^d/(\|x\|-r)^{d-2}\rightarrow 0.
$$
To sum up, we get
$$
\|x\|^{d-4}\sum_{y\in \Z^d}G_K(x,y)\rS(y)\sim a_d^2\sigma^2\BCap(K)\cdot t_d/2.
$$
This completes the proof of \eqref{q-lim}.

\eqref{q_lim} can be obtained in a very similar way and we leave the details to the reader. Note that one shall care about whether to use the original walk and the reversed walk. For example, instead of \eqref{for-q1}, we have:
\begin{equation}\label{for-qq}
\qS^-(x)=\sum_{y\in \Z^d}g(y,x)\rS(y)\EsC_K(y).
\end{equation}
\end{proof}

\begin{remark}
The analogous result also holds for branching random walk conditioned on survival and can be proved similarly. If we write $\overline{\Snake}^\infty_x$ for the branching random walk conditioned on survival starting from $x$, then we have:
\begin{equation}
\lim_{x\rightarrow \infty}\|x\|^{d-4}\cdot P(\overline{\Snake}^\infty_x\text{ visits K})=t_d a_d^2\sigma^2\BCap(K).
\end{equation}
\end{remark}

\section{Convergence of the conditional entering measure.}

Theorem \ref{MT1} implies that conditioned on visiting a finite set, the first visiting point and the last visiting point converge in distribution as the starting point tends to infinity. In fact, not only the first and last visiting points, but also the set of 'entering' points converge in distribution. Let us make this statement precise.

As before, we fix a $K\ssubset \Z^d$. Let $\Mp(K)$ stand for the set of all finite point measures on $K$. The entering measure of a finite snake $\Snake_x=(T, \Snake_T)$ is defined by:
$$
\Theta_x=\sum_{v\in T:\Snake _T(v)\in K, v \text{ has no ancestor lying in } K}\delta_{\Snake _T(v)}.
$$
Note that $\Theta_x$ is a random element in $\Mp(K)$ and
$$
P(\Theta_x\neq 0)=P(\Snake_x \text{ visits } K)\leq c_K |x|^{2-d},\quad
E(\langle \Theta_x,1\rangle)\leq g(x, K)\leq c_K |x|^{2-d}.
$$
We write $\widetilde{\Theta}_x$ for $\Theta_x$ conditioned on the initial particle having exactly one child and $\overline{\Theta}_x$ for $\Theta_x$ conditioned on $\Theta_x\neq0$.
Now we can state our result:
\begin{theorem}\label{T-hm}
\begin{equation}
\overline{\Theta}_x\stackrel{d}{\rightarrow}\hm_K, \;\text{as }|x|\rightarrow \infty,
\end{equation}
where $\hm_K$ is defined later in \eqref{def-hm} and $\stackrel{d}{\rightarrow}$ means convergence in distribution.
\end{theorem}

\begin{remark}
We will introduce many notations, later this section, which will be only used for the proof of Theorem \ref{T-hm}. The reader may wish to skip this section at first reading. All results in this section are not needed for the rest of the paper.
\end{remark}

\subsection{Construction of the limiting measure.}
There are two steps needed, to sample an element from $\hm_K$.  The first step is to sample the 'left-most' path ($\bb(\Snake_x)$) appeared in Section 5 and then run independent branching random walks from all vertices on that path.

We begin with the second step. Inspired by \eqref{key1}, we introduce the position-dependent distribution $\mu_x$ on $\N$ and random variable $\Lambda_x$ on $\Mp(K)$:
\begin{equation}\label{mu-x}
\mu_x(m)=\sum_{l\geq0}\mu(l+m+1)(\rp(x))^l/(1-\rS(x)),\quad \text{for } x\notin K,
\end{equation}
$$
\Lambda_x\stackrel{d}{=}
\left\{
\begin{array}{ll}
\Sigma_{i=1}^{N}X_i, &\text{when }x\notin K;\\
\delta_x, &\text{when }x\in K;\\
\end{array}
\right.
$$
where $N$ is an independent random variable with distribution $\mu_x$ and $X_i$ are i.i.d. with distribution $\widetilde{\Theta}_x$. Note that
$$
\lim_{x\rightarrow\infty}\mu_x(m)= \tilde{\mu}(m), \quad
\mu_x(m)\leq \tilde{\mu}(m)/(1-\rS(x))\preceq \tilde{\mu}(m),
$$
$$
E\langle\Lambda_x,1\rangle=(E\mu_x)E(\langle \Theta_x,1\rangle)\leq c_K |x|^{2-d}.
$$

For any path $\gamma$, define $\Zm(\gamma)$ and $\Zmm(\gamma)$ by:
$$
\Zm(\gamma)=\Sigma_{i=0}^{|\gamma|}\Lambda(i); \Zmm(\gamma)=\Sigma_{i=0}^{|\gamma|-1}\Lambda(i),
$$
where $\Lambda(i)\stackrel{d}{=}\Lambda_{\gamma(i)}$ are independent random variables. Note that
$$
E\langle\Zm(\gamma),1\rangle\leq c_K\Sigma_{i=0}^{|\gamma|}|\gamma(i)|^{2-d}.
$$
Hence, for an infinite path $\gamma:\N\rightarrow\Z^d$, we can also define $\Zm(\gamma)$:
$$
\Zm(\gamma)=\sum_{i=0}^{\infty}\Lambda(\gamma(i))\in \Mp(K) \;a.s.,
$$
as long as
\begin{equation}\label{con-inf}
\sum_{i=0}^{\infty}|\gamma(i)|^{2-d}<\infty.
\end{equation}

Now we move to the first step and explain how to sample the left-most path. For any $x\in \Z^d$, let $h(x)=P(\Snake^-_x \text{ does not visit } K)$. Define $P^\infty$ to be the transition probability of the Markov chain in $\{z\in\Z^d:\EsC_K(z)>0\}$ by:
\begin{equation}
P^\infty(x,y)=\frac{\theta(x-y)h(y)}{\sum_{z\in\Z^d}\theta(x-z)h(z)}=\frac{\theta(x-y)(1-\KRW(y))\EsC_K(y)}{\EsC_K(x)}.
\end{equation}
For any $x$ with $\EsC_K(x)>0$, define $P_x^\infty$ to be the law of random walk starting from $x$ with transition probability $P^{\infty}$.
Define $P^{\infty}_K$ to be the law of random walk (with transition probability $P^{\infty}$) starting at $a\in K$ with probability $\EsC_K(a)/\BCap(K)$.

Now we can give the definition of $\hm_K$:
\begin{equation}\label{def-hm}
\hm_K=\mathrm{~the~law~of~}Z:\mathrm{~where~first~sample~}\gamma\mathrm{~by~}
P^\infty_K\mathrm{~and~then~sample~}Z\mathrm{~by~}\Zm(\gamma).
\end{equation}
Note that under $P^\infty_x$ (for those $x$ with $\EsC_K(x)>0$),
\begin{align*}
E_x^\infty&\sum_{i=0}^{\infty}|\gamma(i)|^{2-d}=E_x^\infty\sum_{z\in \Z^d}
\sum_{i\in\N}\mathbf{1}_{\gamma(i)=z}|z|^{2-d}=\sum_{z\in \Z^d}|z|^{2-d}E_x^\infty
\sum_{i\in\N}\mathbf{1}_{\gamma(i)=z}\\
&=\sum_{z\in \Z^d}|z|^{2-d}\frac{G_K(z,x)\EsC_K(z)}{\EsC_K(x)}\preceq
\frac{1}{\EsC_K(x)}\sum_{z\in \Z^d}|z|^{2-d}|z-x|^{2-d}\preceq\frac{|x|^{4-d}}{\EsC_K(x)}<\infty.\\
\end{align*}
Therefore, under $P_x^{\infty}$ (and hence $P_K$), $\Zm(\gamma)$ is well-defined a.s..

\subsection{Convergence of the conditional entering measure.}
Since our sample space $\Mp(K)$ is discrete and countable, it is convenient to use the total variation distance. Recall that for two probability distributions $\nu_1,\nu_2$ on a discrete countable space $\Omega$, the total variation distance is defined to be
\begin{equation*}
d_{TV}(\nu_1,\nu_2)=\frac{1}{2}\sum_{\omega\in\Omega}|\nu_1(\omega)-\nu_2(\omega)|\in[0,1]
\end{equation*}
and $\nu_n\stackrel{d}{\rightarrow}\nu$ iff $d_{TV}(\nu_n,\nu)\rightarrow0$.

Let us introduce some notations. Let $\Gamma$ be a countable set of finite paths. For each $\gamma\in\Gamma$, assign to it, the weight $a(\gamma)\geq0$ (assume that the total mass $\sum_{\gamma\in\Gamma}a(\gamma)\leq1$) and a probability law $Z(\gamma)$ in $\Mp(K)$. We denote by $\bigsqcup_{\gamma\in \Gamma}a(\gamma)\cdot Z(\gamma)$ for the random element in $\Mp(K)$ as follows: pick a random path $\gamma'$ among $\Gamma$ with probability $P(\gamma'=\gamma)=a(\gamma)$ (with probability $1-\sum_{\gamma\in\Gamma}a(\gamma)$ we do not get any path and in this case simply set $\bigsqcup_{\gamma\in \Gamma}a(\gamma)\cdot Z(\gamma)=0$) and then use the law $Z(\gamma')$ to sample $\bigsqcup_{\gamma\in \Gamma}a(\gamma)\cdot Z(\gamma)$. We also write $\bigsqcup_{\gamma\in \Gamma}a(\gamma)\cdot Z(\gamma)$ for its law (one can judge by the text when the notation appears). One can easily verify the following proposition:
\begin{prop}\label{prop-TV}
If $\nu=\bigsqcup_{\gamma\in \Gamma}a(\gamma)\cdot Z(\gamma)$, $\nu_1=\bigsqcup_{\gamma\in \Gamma}a_1(\gamma)\cdot Z(\gamma)$ and $\nu_2=\bigsqcup_{\gamma\in \Gamma}a(\gamma)\cdot Z_1(\gamma)$, then
\begin{equation}
d_{TV}(\nu,\nu_1)\leq\sum_{\gamma\in \Gamma}|a(\gamma)-a_1(\gamma)|,\quad
d_{TV}(\nu,\nu_2)\leq\sum_{\gamma\in \Gamma}a(\gamma)d_{TV}(Z(\gamma),Z_1(\gamma)).
\end{equation}
\end{prop}
For any $n>\Rad(K)$, write:
$$
\hm_K^n=\bigsqcup_{\gamma:(\Ball(n))^c\rightarrow K,\gamma\subseteq (\Ball(n)\setminus K)}\frac{\BRW(\gamma)\EsC_K(\gamma(0))}{\BCap(K)}\cdot\Zm(\gamma).
$$
Note that $\hm_K^n$ can be defined equivalently as follows: first sample a infinite path $\gamma'$ by $P^{\infty}_K$ and cut $\gamma'$ into two pieces at the hitting time of $(\Ball(n))^c$; let $\gamma$ be the first part and then sample $\hm_K^n$ by $\Zm(\gamma)$.  Hence, we have: $\hm_K^n\stackrel{d}{\rightarrow}\hm_K$ as $n\rightarrow\infty$.

Now we turn to $\overline{\Theta}_x$. Similar to the computations after \eqref{key1}, one can get, for $\gamma=(\gamma(0),\dots, \gamma(k))\subseteq K^c$ with $\gamma(0)=x, \widehat{\gamma}=\gamma(k)\in K$ and $1\leq j_1< j_2\leq k$, (see the corresponding notations there)
\begin{equation*}
P(\widetilde{b}_{j_1}=m|\Gamma(\Snake_x)=\gamma)=
\frac{\sum_{l\in\N}\mu(l+m+1)(\rp(\gamma(j_1-1)))^l}{1-\rS(\gamma(j_1-1))};
\end{equation*}
\begin{multline*}
P(\widetilde{b}_{j_i}=m_i, \text{for }i=1,2|\Gamma(\Snake_x)=\gamma)=\\
\frac{\sum_{l\in\N}\mu(l+m_1+1)(\rp(\gamma(j_1-1)))^l}{1-\rS(\gamma(j_1-1))}
\frac{\sum_{l\in\N}\mu(l+m_2+1)(\rp(\gamma(j_2-1)))^l}{1-\rS(\gamma(j_2-1))}.
\end{multline*}
From these (and the similar equations for more than two $b_j$'s), one can see that conditioned on the event $\Gamma(\Snake_x)=\gamma$, $(\widetilde{b}_j)_{j=1,\dots,k}$ are independent and have the distribution of the form in \eqref{mu-x}. Hence, conditioned on $\Gamma(\Snake_x)=\gamma$, $\Theta_x$ has the law of $\Zm(\gamma)$. Therefore, we have
\begin{prop}
\begin{equation}
\overline{\Theta}_x=\bigsqcup_{\gamma:x\rightarrow K}\frac{\BRW(\gamma)}{\pS(x)}\cdot \Zm(\gamma).
\end{equation}
\end{prop}

Set $n=n(x)=\|x\|^{\frac{d-1}{d}}$. We need to show:
\begin{equation}
\lim_{x\rightarrow\infty}d_{TV}(\overline{\Theta}_x,\hm_K^n)=0.
\end{equation}

Let $B=\Ball(n)$ and $B_1=\Ball(2n)$. For any $\gamma:x\rightarrow K$, we decompose $\gamma$ into two pieces $\gamma=\gamma_1\circ \gamma_2$ according to the last visiting time of $B^c$. We can rewrite $\overline{\Theta}_x$ as follows:
\begin{align*}
\overline{\Theta}_x=&\bigsqcup_{\gamma:x\rightarrow K}\frac{\BRW(\gamma)}{\pS(x)}\cdot \Zm(\gamma)
=\bigsqcup_{\gamma:x\rightarrow K}\frac{\BRW(\gamma_1)\BRW{(\gamma_2)}}{\pS(x)}\cdot (\Zm^-(\gamma_1)+\Zm(\gamma_2))\\
=&\bigsqcup_{\gamma_2:B^c\rightarrow K, \gamma_2\subseteq B}\frac{\BRW(\gamma_2)g_K(x,\gamma_2(0))
}{\pS(x)}\cdot
(\Zm(\gamma_2)+\bigsqcup_{\gamma_1:x\rightarrow \gamma_2(0)}\frac{\BRW(\gamma_1)}
{g_K(x,\gamma_2(0))}\cdot\Zm^-(\gamma_1))
\end{align*}
We point out that
\begin{equation}\label{tt1}
\sum_{\gamma_2:B_1^c\rightarrow K, \gamma_2\subseteq B}\frac{\BRW(\gamma_2)g_K(x,\gamma_2(0))
}{\pS(x)}\ll 1.
\end{equation}
This can be seen from: (by the Overshoot Lemma and \eqref{green})
$$
\sum_{\gamma_2:\Ball(\|x\|/2)^c\rightarrow K, \gamma_2\subseteq B}
\BRW(\gamma_2)g_K(x,\gamma_2(0))\preceq n^2/\|x\|^d\ll \pS(x);
$$
$$
\sum_{\gamma_2:\Ball(\|x\|/2)\setminus B_1\rightarrow K, \gamma_2\subseteq B}
\BRW(\gamma_2)g_K(x,\gamma_2(0))\preceq n^2/n^d\cdot \|x\|^{2-d}\ll \pS(x).
$$
Furthermore, when $y\in B_1$, by \eqref{Conver-G}, \eqref{green} and \eqref{m1-1}, we have:
\begin{equation}\label{tt2}
\frac{g_K(x,y)}{\pS(x)}\sim \frac{1}{\BCap(K)}\sim \frac{\EsC_K(y)}{\BCap(K)}.
\end{equation}
Hence (by Proposition \ref{prop-TV}, \eqref{tt1} and \eqref{tt2}), we have:
$$
d_{TV}\left(\overline{\Theta}_x,\bigsqcup_{\gamma:B_1\setminus B\rightarrow K, \gamma\subseteq B}\frac{\BRW(\gamma)
\EsC_K(\gamma(0))}{\BCap(K)}\cdot
(\Zm(\gamma)+\bigsqcup_{\gamma_1:x\rightarrow \gamma(0)}\frac{\BRW(\gamma_1)}
{g_K(x,\gamma(0))}\cdot\Zm^-(\gamma_1))\right)\rightarrow0.
$$
Similarly, we have:
$$
d_{TV}\left(\hm_K^n,\bigsqcup_{\gamma:B_1\setminus B\rightarrow K,\gamma\subseteq B}\frac{\BRW(\gamma)\EsC_K(\gamma(0))}{\BCap(K)}\cdot\Zm(\gamma)\right)\rightarrow0.
$$
On the other hand, for any $\gamma:B_1\setminus B\rightarrow K,\gamma\subseteq B$, we have (let $y=\gamma(0)$):
\begin{align*}
d_{TV}&\left(\Zm(\gamma)+\bigsqcup_{\gamma_1:x\rightarrow \gamma(0)}\frac{\BRW(\gamma_1)}
{g_K(x,\gamma(0))}\cdot\Zm^-(\gamma_1), \Zm(\gamma)\right)\\
&\leq P(\bigsqcup_{\gamma_1:x\rightarrow y}\frac{\BRW(\gamma_1)}
{g_K(x,y)}\cdot\Zm^-(\gamma_1)\neq0)\leq E\langle\bigsqcup_{\gamma_1:x\rightarrow y}\frac{\BRW(\gamma_1)}
{g_K(x,y)}\cdot\Zm^-(\gamma_1),1\rangle\\
&\preceq\sum_{\gamma_1:x\rightarrow y}\frac{\BRW(\gamma_1)}{g_K(x,y)}\sum_{i=0}^{|\gamma_1|}|\gamma_1(i)|^{2-d}
=\sum_{\gamma_1:x\rightarrow y}\frac{\BRW(\gamma_1)}{g_K(x,y)}\sum_{z\in\Z^d}|z|^{2-d}\sum_{i=0}^{|\gamma_1|}
\mathbf{1_{\gamma_1(i)=z}}.\\
\end{align*}
Note that $g_K(x,y)\sim |x|^{2-d}$ and $\BRW(\gamma)\leq\SRW(\gamma)$, it suffices to show:(when $x\rightarrow\infty$, uniformly for any $y\in B_1\setminus B$)
\begin{equation}\label{o2}
|x|^{d-2}\sum_{z\in\Z^d}|z|^{2-d}\sum_{\gamma_1:x\rightarrow y}\SRW(\gamma_1)\sum_{i=0}^{|\gamma_1|}\mathbf{1_{\gamma_1(i)=z}}\rightarrow0.
\end{equation}
Note that
$$
\sum_{\gamma_1:x\rightarrow y}\SRW(\gamma_1)\sum_{i=0}^{|\gamma_1|}\mathbf{1_{\gamma_1(i)=z}}
=\sum_{\gamma_3:x\rightarrow z}\SRW(\gamma_3)\sum_{\gamma_4:z\rightarrow y}\SRW(\gamma_4)=g(x,z)g(z,y).
$$
Hence, the left hand side of \eqref{o2} can be bounded by:
\begin{align*}
&|x|^{d-2}\sum_{z\in\Z^d}|z|^{2-d}g(x,z)g(z,y)\preceq|x|^{2-d}\sum_{z\in\Z^d}|z|^{2-d}|x-z|^{2-d}|y-z|^{2-d}\\
&=|x|^{d-2}(\sum_{z:|z-x|\leq|x|/2}+\sum_{z:|z-x|>|x|/2})|z|^{2-d}|x-z|^{2-d}|y-z|^{2-d}\\
&\preceq|x|^{d-2}(\sum_{z:|z-x|\leq|x|/2}|x|^{2-d}|z-x|^{2-d}|x|^{2-d}
+\sum_{z:|z-x|>|x|/2}|z|^{2-d}|x|^{2-d}|y-z|^{2-d})\\
&\preceq|x|^{d-2}(|x|^{6-2d}+|y|^{4-d}|x|^{2-d})\asymp |y|^{4-d}\preceq n^{4-d}\rightarrow 0.
\end{align*}
This completes the proof of \eqref{o2} and hence Theorem \ref{T-hm}.

\section{Branching capacity of balls.}
In this section, we compute the branching capacity of balls. As mentioned before, we carry out this by estimating the visiting probability of balls and then use \eqref{m1-1} in reverse. Let us set up the notations. For $x\in\Z^d$ and $A\ssubset \Z^d$, we write $\pS_A(x)$, $\rS_A(x)$, $\qS_A(x)$ and $\qS^-_A(x)$ respectively, for the probability that a snake, an adjoint snake, an infinite snake and a reversed snake respectively, starting from $x$ visits $A$.
\begin{theorem}\label{lem-visiting-finite}
Let $A=\{z=(z_1,0)\in\Z^m\times\Z^{d-m}=\Z^d: \|z\|\leq r\}$ be the $m$-dimensional ball ($1\leq m\leq d$) with radius $r\geq1$ and $x\in \Z^d\setminus A$. When $s=\dist (x,A)\geq2$, we have
\begin{equation*}
\pS_A(x)\asymp\left\{\begin{array}{ll}
r^{d-4}/s^{d-2},&\text{if }m\geq d-3 \text{ and } s\geq r;\\
1/s^2,&\text{if }m\geq d-3 \text{ and } s\leq r;\\
r^{d-4}/(s^{d-2}\log r),&\text{if }m= d-4 \text{ and } s\geq r;\\
1/(s^{2}\log s),&\text{if }m= d-4 \text{ and } s\leq r;\\
r^{m}/s^{d-2},&\text{if }m\leq d-5 \text{ and } s\geq r;\\
1/s^{d-m-2},&\text{if }m\leq d-5 \text{ and } s\leq r.\\
\end{array}
\right.
\end{equation*}
\end{theorem}

\begin{proof}
Let us first mention the organization of the proof. All lower bounds will be proved by the second moment method. So we first
estimate the first and the second moments. For upper bounds, due to Markov property (from Proposition \ref{key-1}), the case for 'big $s$' (i.e. $s\geq r$) can be reduced to the case for 'small $s$' (i.e. $s\leq r$). For small $s$, visiting a large $m$-dimensional ball in $\Z^d$ behaves like visiting a point in $\Z^{d-m}$. Hence we can use the results on the latter case.

\emph{\textbf{Upper bounds for $\mathbf{m\leq d-5}$ and lower bounds for all cases.}}
Let $N$ be the number of times the branching random walk visits $A$. Then $\E N=\sum_{z\in A}g(x,z)= g(x,A)$. For the first moment, we point out:
\begin{equation}\label{1M}
g(x,A)\asymp\left\{
\begin{array}{ll}
r^{m}/s^{d-2},&\text{ for }s\geq r;\\
1/s^{d-m-2},&\text{ for } s\leq r, m\leq d-3 .\\
\end{array}
\right.\\
\end{equation}
The proof is straightforward. When $s\geq r$, for any $a\in A$, $\dist(x,a)\asymp s$. Hence $g(x,A)\asymp |A|\cdot 1/s^{d-2}\asymp r^{m}/s^{d-2}$. When $s\leq r, m\leq d-3$, the part of $\succeq$ is easy. Let $b\in A$ satisfying $\dist(x,b)=\dist(x,A)$ and let $B=b+\Ball(s)$. Then for any $a\in B\cap A$, $\dist(x,a)\asymp s$ and $|B\cap A|\asymp s^m$. Hence $g(x,A)\succeq s^{2-d}\cdot s^m=1/s^{d-m-2}$. For the other part, it needs a bit more work. Assume $x=(\bar{x}_1,\bar{x}_2)\in \Z^m\times\Z^{d-m}$ and let $x_1=(\bar{x}_1,0), x_2=(0,\bar{x}_2)\in \Z^d$. Since $\dist(x,A)=s$, either $\dist(x,x_1)\geq s/2$ or $\dist(x_1,A)\geq s/2$. When $s/2 \leq \dist(x,x_1)=\|x_2\|$, note that $|x_2|\succeq\|x_2\|\succeq s$. We have:
\begin{align*}
g(x,A)\leq& \sum_{z\in \Z^m\times 0\subseteq \Z^d}g(x,z)
\asymp \sum_{z_1\in \Z^m}(\sqrt{|z_1|^2+|x_2|^2})^{2-d}\\
=& \sum_{z_1\in \Z^m,|z_1|\leq s}(\sqrt{|z_1|^2+|x_2|^2})^{2-d}+
\sum_{z_1\in \Z^m,|z_1|\geq s}(\sqrt{|z_1|^2+|x_2|^2})^{2-d}\\
\leq&\sum_{z_1\in \Z^m,|z_1|\leq s}|x_2|^{2-d}+
\sum_{z_1\in \Z^m,|z_1|\geq s}(|z_1|)^{2-d}\\
\preceq& s^m \cdot s^{2-d}+\sum_{n\geq s}\frac{n^{m-1}}{n^{d-2}}=s^{m+2-d}+\sum_{n\geq s}\frac{1}{n^{d-m-1}}\\
\asymp& 1/s^{d-m-2}.\\
\end{align*}
When $\dist(x_1,A)\geq s/2$, note that $|x_1|\asymp\|x_1\|\succeq s$. We have:
\begin{align*}
g(x,A)\asymp& \sum_{z\in \Z^m\times 0\subseteq \Z^d,\|z\|\leq r}\dist(x,z)^{2-d}
\leq \sum_{z\in \Z^m\times 0\subseteq \Z^d,\|z-x_1\|\geq s/2}\dist(x,z)^{2-d}\\
=& \sum_{z\in \Z^m\times 0\subseteq \Z^d,\|z\|\geq s/2}\|z\|^{2-d}\asymp
\sum_{z\in \Z^m\times 0\subseteq \Z^d,\|z\|\geq s/2}|z|^{2-d}\\
\leq&\sum_{n\geq Cs}\frac{n^{m-1}}{n^{d-2}}=\sum_{n\geq Cs}\frac{1}{n^{d-m-1}}\asymp1/s^{d-m-2}.\\
\end{align*}
Now we finish the proof of \eqref{1M}. Note that \eqref{1M} is also true even for $x\in A$ i.e. $g(x,A)\asymp 1$(recall that since we
set $\|0\|=1/2$, when $x\in A$, $\dist(x,A)=1/2$ by our convention).

Using $P(N>0)\leq \E N$, one can get the desired upper bounds for $m\leq d-5$.

For the lower bounds, we need to estimate the second moment and the following is a standard result for branching random walk (for example, see Remark 2 in Page 13 of \cite{LL141}).
\begin{lemma}There exists a constant $C$, such that:
\begin{equation*}
\E N^2\leq C \sum_{z\in \Z^d}g(x,z)g^2(z,A).
\end{equation*}
\end{lemma}

First consider the case when $A$ is a $m$-dimensional ball and $m\leq d-3, s\geq r$. Let $B_0=\{z\in \Z^d: \dist(z,A)\leq r/6\}$
and $B_n=\Ball(2^{n}s/3)$, for $n\in \N^+$. Note that there exists some $c>0$, such that $\BallE(c^{-1}r)\subseteq B_0\subseteq \BallE(cr)$ and $\BallE(c^{-1}2^{n-1}s)\subseteq B_n\subseteq\BallE(c2^{n-1}s)$, for any $n\geq1$. We will divide the sum into three parts and estimate separately: $
\sum_{z\in B_0}$, $\quad \sum_{z\in B_1\setminus B_0}$, and $\sum_{n\geq2}\sum_{z\in B_n\setminus B_{n-1}}$.
When $z=(z_1,z_2)\in B_0$, where $z_1\in\Z^m$ and $z_2\in\Z^{d-m}$, $\|x-z\|\asymp s$ and $\dist(z,A)\succeq |z_2|$.
Hence
\begin{align*}
\sum_{z\in B_0}g(x,z)g^2(z,A)&\asymp\sum_{z\in B_0}\frac{1}{\|x-z\|^{d-2}}\frac{1}{\rho(z,A)^{2d-2m-4}}
\preceq\frac{1}{s^{d-2}}\sum_{z\in \BallE(cr)}\frac{1}{|z_2|^{2d-2m-4}}\\
&\preceq\frac{1}{s^{d-2}} \sum_{|z_1|\leq cr,|z_2|\leq cr}\frac{1}{|z_2|^{2d-2m-4}}
\asymp\frac{r^m}{s^{d-2}}\sum_{|z_2|\leq cr}\frac{1}{|z_2|^{2d-2m-4}}\\
&\asymp\frac{r^m}{s^{d-2}}\sum_{n\leq cr}\frac{n^{d-m-1}}{n^{2d-2m-4}}
=\frac{r^m}{s^{d-2}}\sum_{n\leq cr}\frac{1}{n^{d-m-3}}\\
&\asymp\left\{
\begin{array}{ll}
r^{m+1}/s^{d-2},&\text{ if }m=d-3;\\
r^{m}\log r/s^{d-2},&\text{ if }m=d-4;\\
r^m/s^{d-2},&\text{ if }m\leq d-5.\\
\end{array}
\right.\\
\end{align*}
When $z\in B_1\setminus B_0$, $\|x-z\|\asymp s$, $\dist(z,A)\asymp |z|$ and $g(z,A)\asymp r^m/|z|^{d-2}$. Hence:
\begin{align*}
\sum_{z\in B_1\setminus B_0}&g(x,z)g^2(z,A)\asymp\sum_{z\in B_1\setminus B_0}\frac{1}{\|x-z\|^{d-2}}\left(\frac{r^m}{|z|^{d-2}}\right)^2\\
&\preceq\sum_{z\in \BallE(cs)\setminus \BallE(c^{-1}r)}\frac{r^{2m}}{s^{d-2}|z|^{2d-4}}
\preceq\frac{r^{2m}}{s^{d-2}}\sum_{c^{-1}r\leq n\leq cs}\frac{n^{d-1}}{n^{2d-4}}
\asymp\frac{r^{2m}}{s^{d-2}}\frac{1}{r^{d-4}}.\\
\end{align*}
Note that this term is not bigger than the first term and hence negligible. The remaining part can be estimated similarly and is also negligible:
\begin{align*}
\sum_{n\geq2}&\sum_{z\in B_n\setminus B_{n-1}}g(x,z)g^2(z,A)
\asymp \sum_{n\geq2}\sum_{z\in B_n\setminus B_{n-1}}
\frac{1}{|x-z|^{d-2}}\left(\frac{r^m}{|z|^{d-2}}\right)^2\\
&\preceq\sum_{n\geq2}\sum_{z\in \BallE(c2^{n-1}s)\setminus \BallE(c^{-1}2^{n-1}s)}\frac{1}{|x-z|^{d-2}}\frac{r^{2m}}{(2^ns)^{2d-4}}
\stackrel{(\ast)}{\preceq}\sum_{n\geq2}\frac{(2^ns)^2r^{2m}}{(2^ns)^{2d-4}}\\
&=\sum_{n\geq2}\frac{r^{2m}}{s^{2d-6}}\frac{1}{(2^n)^{2d-6}}
\asymp\frac{r^{2m}}{s^{2d-6}}\leq\frac{r^{2m}}{s^{d-2}}\frac{1}{r^{d-4}}.\\
\end{align*}
$(\ast)$ is due to the fact that
$\sum_{z\in \BallE(n)}|x-z|^{2-d}\leq\sum_{z\in \BallE(n)}|z|^{2-d}\asymp n^2$.

To summarize,  we get:
$$
\sum_{z\in \Z^d}g(x,z)g^2(z,A)\preceq\left\{
\begin{array}{ll}
r^{m+1}/s^{d-2,}&\text{ if }m=d-3;\\
r^{m}\log r/s^{d-2},&\text{ if }m=d-4;\\
r^m/s^{d-2},&\text{ if }m\leq d-5.\\
\end{array}
\right.\\
$$

For $r\geq s$, since we are considering lower bound now, by monotonicity, we can assume $m\leq d-3, r\in[s/2,s]$. Then, we can just let $r\asymp s$ in the last formula and get:
\begin{equation*}
\sum_{z\in \Z^d}g(x,z)g^2(z,A)\preceq\left\{
\begin{array}{ll}
s^{m+1}/s^{d-2}=1,&\text{ if }m=d-3;\\
s^{m}\log s/s^{d-2}=\log s/s^2,&\text{ if }m=d-4;\\
s^m/s^{d-2}=1/s^{d-m-2},&\text{ if }m\leq d-5.\\
\end{array}
\right.\\
\end{equation*}
Using $\pS_A(x)=P(N>0)\geq (\E N)^2/\E N^2$, one can get the required lower bounds for all cases.

\emph{\textbf{From small $\mathbf{s}$ to big $\mathbf{s}$.}}We have proved the upper bound for $m\leq d-5$ and now consider the case $m\geq d-4$. Assume that we have the desired upper bounds for small $s$. We want the upper bound for  big $s$. Let
$B=\{z\in\Z^d: \dist(z,A)\leq r/2\}$ and $C=\{z\in\Z^d: \dist(z,A)\leq r/4\}$. Then by the assumption, we know that for any $z\in B\setminus C$, $\pS(z)\preceq \alpha(r)$, where $\alpha(r)=1/r^2$ or $1/(r^2\log r)$ depending on $m$. Let
$$
\Gamma_1=\{\gamma:x\rightarrow A|\gamma\subseteq A^c,\gamma\text{ visits } B\setminus C\},
$$
$$
\Gamma_2=\{\gamma:x\rightarrow A| \gamma\subseteq A^c,\gamma\text{ avoids } B\setminus C\}.
$$
We decompose $\pS_A(x)$ into two pieces:
$$
\pS_A(x)=\sum_{\gamma:x\rightarrow A}\BRW(\gamma)=\sum_{\gamma\in \Gamma_1}\BRW(\gamma)+\sum_{\gamma\in \Gamma_2}\BRW(\gamma).
$$
For the first term, by considering the first visiting point of $B\setminus C$, one can see:
\begin{align*}
\sum_{\gamma\in \Gamma_1}\BRW(\gamma)&\leq
\sum_{z\in B\setminus C} \, \sum_{\gamma:x\rightarrow z,\gamma\subseteq (B\setminus C)^c }\BRW(\gamma)\pS_A(z)\leq \sum_{z\in B\setminus C} \, \sum_{\gamma:x\rightarrow z,\gamma\subseteq (B\setminus C)^c }
\SRW(\gamma) \alpha(r)\\
&\leq \alpha(r)P(S_x\text{ visits }(B\setminus C))\leq \alpha(r)P(S_x\text{ visits }B)\\
&\asymp \alpha(r) (r/s)^{d-2}=\left\{\begin{array}{ll}
r^{d-4}/s^{d-2}&\text{if }m\geq d-3;\\
r^{d-4}/(s^{d-2}\log r) &\text{if }m=d-4.\\
\end{array}
\right.
\end{align*}
Recall that $S_x$ is the random walk starting from $x$ and we use the standard estimate of
$P(S_x\text{ visits }B)\asymp(r/s)^{d-2}$.
For the other term, by considering the first jump from $B^c$ to $C$, one can see:
\begin{align*}
\sum_{\gamma\in \Gamma_2}&\BRW(\gamma)\leq\sum_{z\in B^c}G_A(x,z)
\sum_{y\in C}\theta(y-z)\pS_A(y)\\
&=\sum_{z\in B_1^c}G_A(x,z)
\sum_{y\in C}\theta(y-z)\pS_A(y)+\sum_{z\in B_1\setminus B}G_A(x,z)
\sum_{y\in C}\theta(y-z)\pS_A(y),
\end{align*}
where $B_1=\{z:\dist(z,A)\leq r/2+s/4\}$.
Both terms are negligible:
\begin{multline*}
\sum_{z\in B_1^c}G_A(x,z)
\sum_{y\in C}\theta(y-z)\pS_A(y)\preceq \sum_{z\in B_1^c}1\cdot\sum_{y\in C}\theta(y-z)\alpha(r)\\
\leq\sum_{y\in C}\alpha(r)\sum_{z\in B_1^c}\theta(y-z)\preceq\sum_{y\in C}\alpha(r)s^{-d}
\preceq \alpha(r)r^d/s^d \leq \alpha(r)(r/s)^{d-2};
\end{multline*}
\begin{multline*}
\sum_{z\in B_1\setminus B}G_A(x,z)
\sum_{y\in C}\theta(y-z)\pS_A(y)\preceq \sum_{z\in B_1\setminus B}s^{2-d}\sum_{y\in C}\theta(y-z)\alpha(r)\\
\leq s^{2-d}\sum_{y\in C}\alpha(r)\sum_{z\in B^c}\theta(y-z)\preceq s^{2-d}\sum_{y\in C}\alpha(r) r^{-d}
\preceq s^{2-d} \alpha(r) \leq \alpha(r)(r/s)^{d-2}.
\end{multline*}

\emph{\textbf{For small $\mathbf{s}$ and $\mathbf{m\geq d-3}$.}} The upper bound in this case relies on the corresponding bound for one dimensional branching random walk. Let $H$ be a half space, say $H=\{z=(z_1,\dots,z_d)\in\Z^d:z_1\geq n\}$. The probability of visiting $H$ is equivalent to the probability of 1-dimensional branching random walk visiting a half line. The asymptotic behavior of the latter case is known. Le Gall and Lin (Theroem 7 in \cite{LL14}) have proved that in low dimensions ($d\leq 3$),
$$
\lim_{a\rightarrow\infty}\|a\|^2P(\text{ Branching random walk from }0 \text{ visits }a)=2(4-d)/d\sigma^2.
$$
However, this result is under the assumption that $\mu$ has finite exponential moment, which we do not assume. For our purpose, we give a weaker result here under weaker assumptions:

\begin{prop}\label{1-d}
Let $\Snake_x$ be 1-dimensional branching random walk starting from $x\in \Z$, given that the offspring distribution $\mu$ is critical and nondegenerate, and the jump distribution $\theta$ has zero mean and finite second moment, and satisfies $\sum_{i:i\leq -k}\theta(i)\leq Ck^{-4}$ for any $k\in \N^{+}$ and some $C$ (independent of $k$). Then for some large constant $c=c(\theta, \mu)>0$ (independent of $x$), we have: for any $x\in \N^{+}$,
\begin{equation}\label{1_d}
P(\Snake_x \text{ visits }\Z^-)\leq c/|x|^2,
\end{equation}
where $\Z^-=\{0,-1,-2,\dots\}$.
\end{prop}

We postpone the proof of this proposition. Return to $d$ dimension. Since we can find at most $d$ half spaces $H_1,H_2,\dots,H_d$ satisfying: $\dist(x,H_i) \asymp s$ for any $i=1,\dots,d$; and that any path from $x$ to $A$ must hit at least one of $H_i$. Then we have:
$$
\pS_A(x)\leq\sum_{i=1}^d P(\Snake_x \text{ visits }H_i)\preceq d\cdot |s|^{-2}\asymp |s|^{-2}.
$$

\emph{\textbf{For small $\mathbf{s}$ and $\mathbf{m= d-4}$.}} Intuitively when the radius $r$ is large, visiting a $m=d-4$ dimensional ball in $\Z^d$, is similar to visiting a point in $\Z^4$. This is indeed the case. \cite{Z15} gives the desired upper bound for the latter case and the method there also works here with slight modifications. We point out the major differences and leave the details to the reader. On the one hand, one should use $\tilde{g}(\gamma)\DeFine\sum_{i=0}^{|\gamma|-1}g(\gamma(i),A)$ instead of $g(\gamma)$ there. On the other hand, in proving an analogy of Lemma 10.1.2(a) in \cite{LL10}, one might use the stopping times:
$$
\tilde{\xi}^i=\min\{k:\dist(S_k,A)\geq 2^i\}\wedge (\tilde{\xi}^{i-1}+(2^i)^2).
$$
instead of $\xi^i$ there.
\end{proof}

\begin{proof}[Proof of Proposition \ref{1-d}]
Write $p(x)$ for the left hand site in \eqref{1_d}. In order to obtain upper bounds of $p(x)$, we use some of the ideas of \cite{LZ11} (Section 7.1), the techniques from nonlinear difference equations. We will exploit the fact that $p(x)$ satisfies a parabolic nonlinear difference equation and use the comparison principle.

Let $p_n(x)=P(\Snake_x \text{ visits }\Z^- \text{ within the first } n \text{ generations})$. Then $p_n(x)$ is increasing for $n$ and converges to $p(x)$ when $n\rightarrow\infty$. On the other hand, one can easily verify that $p_n(x)$ satisfies the recursive equations:
$$
p_0(x)=\mathbf{1}_{\Z^-}(x); \quad p_n(x)=1 \text{ for }x \in \Z^-;
$$
\begin{equation}\label{de}
p_{n+1}(x)=f(\mathbb{A} p_n(x)), \text{ for }x \in \N^{+};
\end{equation}
where $f(t)=1-\sum_{k\geq0}\mu(k)(1-t)^k$ and $\mathbb{A}$ is the Markov operator for the random walk, that is, for any bounded function $w:\Z\rightarrow \R$, $\mathbb{A}w(x)=\sum_{y\in\Z}\theta(y)w(x+y)$. One can see that $f:[0,1]\rightarrow[0,1-\mu(0)]$ is in $C^1[0,1]\cap C^\infty(0,1]$ with the first 2 derivatives as follows:
$$
f'(t)=\sum_{k\geq1}k\mu(k)(1-t)^{k-1}>0\; \text{for }t\in[0,1);\quad f'(0)=1, f'(1)=\mu(1)\geq0;
$$
$$
f''(t)=-\sum_{k\geq2}k(k-1)\mu(k)(1-t)^{k-2}<0\; \text{for }t\in(0,1).
$$
From these, it is easy to obtain:
$$
\inf_{t\in(0,1]}\frac{t-f(t)}{t^2}>0.
$$
Hence we can find some $a\in(0,1/2)$, such that
\begin{equation}\label{de-a}
f(t)\leq t-at^2,\; \text{for any } t\in[0,1] \text{ and }
t(1+at)\leq 1  \; \text{for any } t\in[0,1-\mu(0)].
\end{equation}

To extract information from \eqref{de}, we will use the following standard comparison principle.
\begin{lemma}
Let $u_n(x)$ and $v_n(x)$ be $\Z\rightarrow[0,1]$,  satisfying
$$
u_n(x)=v_n(x)=1, \quad \text{for any }x\in\Z^- \text{ and } n\in\N;
$$
and
$$
u_{n+1}(x)=f(\mathbb{A}u_n(x)),\quad v_{n+1}(x)\geq f(\mathbb{A}v_n(x)) \quad \text{for any } x\in\N^{+}.
$$
If $v_0(x)\geq u_0(x)$ for all $x$, then
$$
v_n(x)\geq u_n(x)\quad \text{for all } n\in \N^{+} \text{ and } x\in\Z.
$$
\end{lemma}
\begin{proof}
Note that for $n>0$ and $x\in\N^+$:
\begin{align*}
v_n(x)-u_n(x)&\geq f(\mathbb{A}v_{n-1}(x))-f(\mathbb{A}u_{n-1}(x))\\
&\geq \min_{t\in[0,1]} \{f'(t)\}(\mathbb{A}v_{n-1}(x)-\mathbb{A}u_{n-1}(x))\\
&=\min_{t\in[0,1]} \{f'(t)\}\mathbb{A}(v_{n-1}-u_{n-1})(x).
\end{align*}
Since $f'(t)\geq0$, one can use induction to finish the proof.
\end{proof}
Now let $u_n(x)=p_n(x)$ and $v_n(x)=v(x)=1\wedge (c/x^2)$ when $x\in\N^{++}$ for some large $c$ (to be determined later). If we can show
\begin{equation}\label{de1}
v(x)\geq f(\mathbb{A}v(x)) \quad \text{for any } x\in\N^{+},
\end{equation}
then by the lemma above we conclude the proof of Proposition \ref{1-d}.

Let us write down our strategy for choosing $c$. First we fix some $\epsilon\in(0,1/2)$, such that $(1-\mu(0))/(1-\epsilon)^2<1$. Choose $c$ satisfying:
\begin{equation}\label{de-b}
ac^2\geq C/\epsilon^4 +3(E|\theta|^2)c/(1-\epsilon)^4.
\end{equation}
We argue that \eqref{de1} is correct for our choice of $c$. When $c/x^2\geq1-\mu(0)$, \eqref{de1} is obvious since $f(t)\leq 1-\mu(0)$.

Now assume $c/x^2<1-\mu(0)$. Since $f(t)$ is increasing, we need to find an upper bound of $\mathbb{A}v(x)$. We achieve this by decomposing $\mathbb{A}v(x)$ into two pieces and estimating each one separately:
$$
\mathbb{A}v(x)=\sum_{y\in\Z}\theta(y)v(x+y)=\sum_{y\leq -\epsilon x}\theta(y)v(x+y)+\sum_{y>-\epsilon x}\theta(y)v(x+y).
$$
We can use our assumption of $\theta$ to bound the first term:
$$
\sum_{y\leq -\epsilon x}\theta(y)v(x+y)\leq \sum_{y\leq -\epsilon x}\theta(y)\leq C/(\epsilon x)^4;
$$
Using Taylor expansion, the second term can be bounded by:
\begin{align*}
\sum_{y>-\epsilon x}&\theta(y)v(x+y)\leq \sum_{y>-\epsilon x}\theta(y)(v(x)+yv'(x)+\frac{y^2}{2}v''((1-\epsilon)x))\\
&=v(x)\sum_{y>-\epsilon x}\theta(y)+v'(x)\sum_{y>-\epsilon x}\theta(y)y+
v''((1-\epsilon)x)\sum_{y>-\epsilon x}\theta(y)\frac{y^2}{2}\\
&\leq v(x)\cdot 1+v'(x)(-\sum_{y\leq-\epsilon x}\theta(y)y)+v''((1-\epsilon)x)E|\theta|^2/2\\
&\leq v(x)+0+\frac{E|\theta|^2}{2}\cdot\frac{6c}{(1-\epsilon)^4x^4}.\\
\end{align*}
To summarize, we get (let $K=(C/\epsilon^4+3E|\theta|^2c/(1-\epsilon)^4)$):
$$
\mathbb{A}v(x)\leq v(x)+(C/\epsilon^4+3E|\theta|^2c/(1-\epsilon)^4)x^{-4}= v(x)+Kx^{-4}.
$$
Note that by \eqref{de-a} and \eqref{de-b}, we have $v(x)+Kx^{-4}\leq v(x)+a(v(x))^2\leq 1$.  Hence:
\begin{align*}
f(\mathbb{A}v(x))&\leq\mathbb{A}v(x)(1-a\mathbb{A}v(x))\leq (v(x)+Kx^{-4})(1-av(x))\\
&\leq v(x)+Kx^{-4}-a(v(x))^2=v(x)+Kx^{-4}-ac^2x^{-4}\leq v(x).\\
\end{align*}
This completes the proof of \eqref{de1} and hence the proof of Proposition \ref{1-d}.
\end{proof}

For the future use, we give the following upper bound for the visiting probability of a ball by an infinite snake.

\begin{lemma}\label{pro-inf-ball}
Let $A=\Ball(r)$ ($r\geq1$) and $x\in\Z^d$ such that $s=\dist(x,A)\geq r$. Then we have:
\begin{equation}\label{bd-inf}
\qS_A(x) \vee \qrS_A(x)  \preceq (r/s)^{d-4}.
\end{equation}
\end{lemma}
\begin{proof}
Consider a big ball $B=\Ball(1.5r)$. Then
$$
\qS_A(x)\leq P(\text{ backbone visits }B)+P(\text{ backbone avoids }B,\; \Snake_x^{\infty}\text { visits }A).
$$
Since the backbone is just a random walk, the first term is comparable to $(r/s)^{d-2}$, which is less that $(r/s)^{d-4}$. On the other hand, when the backbone does not visit $B$, by considering where the particle is killed, we have:
\begin{align*}
P(&\text{ backbone avoids }B,\; \Snake_x^{\infty}\text { visits }A)
\leq\sum_{z\in B^c}G_A(x,z)\rS_A(z)\preceq \sum_{z\in B^c}g(x,z)\pS_A(z)\\
&\asymp \sum_{z\in B^c} \frac{1}{|x-z|^{d-2}}\frac{r^{d-4}}{|\dist(z,A)|^{d-2}}
\asymp \sum_{z\in B^c} \frac{1}{|x-z|^{d-2}}\frac{r^{d-4}}{|z|^{d-2}}
\asymp \frac{r^{d-4}}{|x|^{d-4}}\asymp (\frac{r}{s})^{d-4}.
\end{align*}
This completes the proof of $\qS_A(x)  \preceq (r/s)^{d-4}$ and similarly one can show $\qS_A^-(x)  \preceq (r/s)^{d-4}$.
\end{proof}

\section{Proof of Theorem \ref{MT5}}
We use an equation approach similar to the proof of Proposition \ref{1-d}. Write
$f_i(t)=1-\sum_{k\geq 0}\mu_i(k)(1-t)^k$, $i=1,2$. We need the following little lemma and postpone its proof.
\begin{lemma}\label{f1-f2}
There is a $C=C(\mu_1,\mu_2)>1$ such that, for all $t\in[0,1]$,
\begin{equation}
f_1((Ct) \wedge 1)\leq (Cf_2(t))\wedge 1.
\end{equation}
\end{lemma}
For any $A\ssubset \Z^d$ fixed, as in the proof of Proposition \ref{1-d}, denote $u_{i,n}(x)$ ($i=1,2$) recursively by:
$$
u_{i,0}(x)=\mathbf{1}_A(x),\; u_{0,n}(a)=1 \; \forall a\in A;\quad
u_{i,n+1}(x)=f_i(\mathbb{A}u_{i,n}(x)) \;\forall a\notin A.
$$
With the help of last lemma, one can see that $u_{1,n}(x)\leq Cu_{2,n}(x)$, for any $n,x$. On the other hand, we know that $u_{i,n}(x)\rightarrow \pS_{i,A}(x)$. Hence we have $\pS_{1,A}(x)\leq C\pS_{2,A}(x)$. Then by Theorem \ref{MT1}, one can get Theorem \ref{MT5}.

\begin{proof}[Proof of Lemma \ref{f1-f2}]
Since $\lim_{t\rightarrow 0}f_2(t)/t=1$, when $C$ is large enough, we have $Cf_2(C^{-1})\geq 1-\mu_1(0)=f_1(1)$.
It suffices to show for $t\in[0,C^{-1}]$,
\begin{equation}\label{o1}
g(t)\doteq Cf_2(t)-f_1(Ct)\geq 0.
\end{equation}
Note that $f_i(0)=0,f_i'(0)=1, f_i''(0)=-\text{Var}(\mu_i), f_i''(t)\leq0$ and $|f_i''(t)|$ is non-increasing. Hence we can find some $C=C(\mu_1,\mu_2)>1$ such that, $C|f_1''(1/2)|\geq 2|f_2''(0)|$ (and $Cf_2(C^{-1})\geq 1-\mu_1(0)$). Then we have
$$
g''(t)=C(f_2''(t)-Cf_1''(Ct))\geq
\left\{\begin{array}{ll}
C|f_2''(0)|, & t\in [0,1/(2C)];\\
-C|f_2''(0)|, & t\in [1/(2C),1/C].\\
\end{array}
\right.
$$
Together with $g(0)= g'(0)=0$, one can get \eqref{o1}.
\end{proof}

\section{Bounds for the Green function.}
The speed of convergence in \eqref{Conver-G} depends on $K$, which maybe not convenient in some cases. For example, by that lemma, we know $G_K(x,y)\geq C_K g(x,y)$ (when $|x|,|y|$ are large), but the constant depends on $K$. The purpose of this section is to build up this type of bounds with constants independent of $K$.

Thanks to lemma \ref{pro-g1}, we have:
\begin{lemma}\label{bd_G-0}
Let $U,V$ be two connected bounded open subset of $\R^d$ such that $\overline{U}\subseteq V$. Then there exists a $C=C(U,V)$ such that if $A_n=nU\cap\Z^d, B_n=nV\cap\Z^d$ then when $n$ is sufficiently large and $K\subseteq B_n^c$, we have
\begin{equation}
G_K(x,y)\geq Cg(x,y) \text{ for any }x,y\in A_n.
\end{equation}
\end{lemma}
\begin{proof}
Since $K$ is outside $B_n$, for any $z\in A_n$, $\dist(z,K)\succeq n$. By Proposition \ref{1-d}, one can see that $\pS_K(x)\preceq \dist(x,K)^{-2}$. Hence $\KRW(z)=\rS_K(z)\asymp \pS_K(z)\preceq n^{-2}$. Then we have, for any $\gamma:x\rightarrow y, \gamma\subseteq A_n, |\gamma| \leq2n^2$, $\BRW(\gamma)/\SRW(\gamma)\geq (1-c/n^2)^{2n^2}\succeq 1$ (provided that $n$ is sufficiently large).
Then we have:
$$
G_K(x,y)\geq\sum_{\gamma:x\rightarrow y, \gamma\subseteq A_n, |\gamma| \leq2n^2}\BRW(\gamma)
\succeq \sum_{\gamma:x\rightarrow y, \gamma\subseteq A_n, |\gamma| \leq2n^2}\SRW(\gamma)
\stackrel{\eqref{pro-b1}}{\succeq} g(x,y).
$$
\end{proof}

Before giving a better form, we turn to the escape probability and prove:
\begin{lemma}\label{lm-escape}
For any $\lambda>0$, there exists a positive $C=C(\lambda)$, such that, for any $A\ssubset \Z^d$ and $x\in\Z^d$ satisfying $\|x\|\geq(1+\lambda) \Rad(A)$, we have:
\begin{equation}\label{bd-escape}
\EsC_A(x)>C.
\end{equation}
\end{lemma}

\begin{proof}
By lemma \ref{pro-inf-ball}, we can find a positive constant $c_1>1$, such that, for any $z\in\Z^d$ with $\|z\|\geq c_1\Rad (A)$, we have $\EsC_A(z)\geq1/2$. Write $r=\Rad (A)$, $B=\Ball(2c_1r)$ and
$D=\Ball(4c_1r)\setminus \Ball(3c_1r)$. Without loss of generality, assume $1+\lambda< c_1/2$ and $\|x\|<c_1r$.  For any $y\in D$, by Lemma \ref{bd_G-0} (let $U=\{x\in \R^d:\|x\|\in (1+\lambda,4c_1)\}$), we have (when $r$ is large):
$G_A(y,x)\asymp g(y,x)\asymp r^{2-d}$.
Applying the First-Visit Lemma, we get:
\begin{equation*}
G_A(y,x)=\sum_{z\in B^c}G_A(y,z)\Hm^{B}_\KRW(z,x).
\end{equation*}
Hence,
$$
\sum_{y\in D}G_A(y,x)=\sum_{y\in D}\sum_{z\in B^c}G_A(y,z)\Hm^{B}_\KRW(z,x).
$$
Note that the left hand side is $\asymp r^d\cdot r^{2-d}=r^2$ and the right hand side is not larger than:
$$
\sum_{y\in D}\sum_{z\in B^c}g(y,z)\Hm^{B}_\KRW(z,x)=\sum_{z\in B^c}\Hm^{B}_\KRW(z,x)\sum_{y\in D}g(y,z)
\preceq \sum_{z\in B^c}\Hm^{B}_\KRW(z,x)\cdot r^2.
$$
This implies $\sum_{z\in B^c}\Hm^{B}_\KRW(z,x)\succeq1$. Therefore we have:
\begin{equation*}
\EsC_A(x)=\sum_{z\in B^c}\Hm^{B}_\KRW(z,x)\EsC_A(z)\geq 1/2\cdot\sum_{z\in B^c}\Hm_A^{B}(x,z)\succeq1,
\end{equation*}
which completes the proof.
\end{proof}

\begin{remark}
In fact we prove \eqref{bd-escape} only when $\Rad(A)$ is large. We ignore the case when $\Rad(A)$ is not large since this can be done by a standard argument as follows. If $\Rad(A)$ is not sufficiently large, there are only finite possibilities of $A$. For each of those $A$, we have already known the asymptotics of $\EsC_A(x)$ ($\lim_{x\rightarrow\infty}\EsC_A(x)=1$). On the other hand, it is obvious that for any $\|x\|>\Rad(A)$, $\EsC_A(x)>0$. Hence we can find some $C(A)>0$ satisfying \eqref{bd-escape}. Since there are finite many $C(A)$'s, we can simply choose $C$ to be the smallest one of those $C(A)$ (together with the one for sufficiently large $A$). We will also omit this type of standard arguments later. In fact, we have done this in the proof of Theorem \ref{lem-visiting-finite}.
\end{remark}

Now we are ready to prove the following bound of Green function:
\begin{lemma}\label{bd_Green}
For any $\lambda>0$, there exists $C=C(\lambda)>0$, such that:
for any $A\ssubset \Z^d$ and $x,y\in \Z^d$ with $\|x\|,\|y\|>(1+\lambda)\Rad (A)$, we have:
\begin{equation}\label{bd-Green}
G_A(x,y)\geq C g(x,y).
\end{equation}
\end{lemma}

\begin{proof}
Without loss of generality, assume $\|x\|\leq \|y\|$. By Lemma \ref{bd_G-0} one can assume $\|y\| >10\|x\|$ and note that under this assumption $g(x,y)\asymp \|y\|^{2-d}$. Let $B=\Ball(\|y\|/2)$ and $C=\Ball(3\|y\|/4)$. For any $z\in C\setminus B$, also by Lemma \ref{bd_G-0}, we have $G_A(y,z)\asymp \|y\|^{2-d}$.
Applying the First-Visit Lemma, we have:
\begin{align*}
G_A(x,y)&=\sum_{z\in B^c}\Hm^{B}_\KRW(x,z)G_A(z,y)\succeq\sum_{z\in C\setminus B}\Hm^{B}_\KRW(x,z)\|y\|^{2-d}\\
&=\|y\|^{2-d}(\sum_{z\in B^c}\Hm^{B}_\KRW(x,z)-\sum_{z\in C^c}\Hm_A^{B}(x,z))\\
&\geq \|y\|^{2-d}(\EsC_A(x)-c\|y\|^2/\|y\|^d),
\end{align*}
where for the last step we use the Overshoot Lemma. Therefore, when $\Rad (A)$ is large enough, by Lemma \ref{lm-escape}, $G_A(x,y)\succeq \|y\|^{2-d}\asymp g(x,y)$.
\end{proof}

\section{Proof of Theorem \ref{bd-finite}.}

\begin{proof}By cutting $A$ into small pieces, it is enough to show \eqref{bd-p} under the assumption of
$\|x\|\geq 3 \Rad(A)$.
Also, as before, we can assume $r=\Rad(A)$ is sufficiently large. Let $B=\Ball(2r)$.

\emph{\textbf{Upper bound.}}
By \eqref{p2} and the First-Visit Lemma, we have
$$
\pS_A(x)=\sum_{y\in A}G_A(x,y)=\sum_{y\in A}\sum_{z\in B^c}G_A(x,z)\Hm^{B}_\KRW(z,y).
$$
We will decompose it into two parts and estimate them separately.

Let $D=\{z\in \Z^d: \dist(z,x)\leq 0.1\dist(x,A)\}$. Note that when $z\in B^c\setminus D$,
$\dist(x,z)\succeq\dist(x,A)$ and $\EsC_A(z)\asymp1$ by Lemma \ref{lm-escape}. Hence,
\begin{align*}
\sum_{y\in A}\sum_{z\in B^c\setminus D}&G_A(x,z)\Hm^{B}_\KRW(z,y)\preceq
\sum_{y\in A}\sum_{z\in B^c\setminus D}\dist(x,A)^{2-d}\Hm^{B}_\KRW(z,y)\\
&\asymp\dist(x,A)^{2-d}\sum_{y\in A}\sum_{z\in B^c\setminus D}\Hm^{B}_\KRW(z,y)\EsC_A(z)\\
&\leq\dist(x,A)^{2-d}\sum_{y\in A}\EsC_A(y)=\dist(x,A)^{2-d}\BCap(A).\\
\end{align*}

When $z\in D$, $\dist(z,B)\asymp\dist(x,A)$. By considering the position where the first jump falls into, we have:
$$
\Hm^{B}_\KRW(z,y)\leq\sum_{w\in B}\theta(w-z)G_A(w,y).
$$
Hence,
\begin{align*}
\sum_{y\in A}&\sum_{z\in D}G_A(x,z)\Hm^{B}_\KRW(z,y)\preceq
\sum_{z\in D}g(x,z)\sum_{w\in B}\theta(w-z)\sum_{y\in A}G_A(w,y)\\
&=\sum_{z\in D}g(x,z)\sum_{w\in B}\theta(w-z)\pS_A(w)
\leq\sum_{z\in D}g(x,z)\sum_{w\in B}\theta(w-z)\\
&\stackrel{\eqref{as1}}{\preceq}\sum_{z\in D}g(x,z)\dist(z, B)^{-d}\asymp\sum_{z\in D}g(x,z)\dist(x,A)^{-d}
\preceq\dist(x,A)^{2-d}.\\
\end{align*}
This completes the proof of the upper bound.

\emph{\textbf{Lower bound.}} First choose some $a>1$, such that for any $s\geq1$,
\begin{equation}\label{ca}
|\Ball(s)|\cdot\theta\{\left(\Ball((a-1)s)\right)^c\}\leq \frac{\BCap(\{0\})}{2},
\end{equation}
Note that our assumption of $\theta$ guarantees that $\theta\{\left(\Ball((a-1)s)\right)^c\}\preceq ((a-1)s)^{-d}$.

Write $\rho=\dist(x,A)$ and let $C=\Ball(a\rho)$. Note that $\rho\geq 2r$ and $\rho(x,B)\geq r$. Hence
$r\leq \rho/2$, $B\subseteq \Ball(\rho)$ and for any $w\in B, z\in C^c$,
\begin{equation}\label{ca2}
\rho(w,z)\geq (a-1)\rho.
\end{equation}

Then
\begin{align*}
\pS_A(x)&=\sum_{y\in A}G_A(x,y)=\sum_{y\in A}\sum_{z\in B^c}G_A(x,z)\Hm^{B}_\KRW(z,y)\\
&\geq \sum_{y\in A}\sum_{z\in C\setminus B}G_A(x,z)\Hm^{B}_\KRW(z,y)
\succeq \sum_{y\in A}\sum_{z\in C\setminus B} (2a\rho)^{2-d}\Hm^{B}_\KRW(z,y),
\end{align*}
We use the last Lemma in the last step.
It is sufficient to show:
\begin{equation}\label{b1}
\sum_{y\in A}\sum_{z\in C\setminus B} \Hm^{B}_\KRW(z,y)\succeq\BCap(A).
\end{equation}
Note that:
\begin{align*}
\sum_{y\in A}&\sum_{z\in C\setminus B} \Hm^{B}_\KRW(z,y)\geq
\sum_{y\in A}\sum_{z\in C\setminus B} \Hm^{B}_\KRW(z,y)\EsC_A(z)\\
=&\sum_{y\in A}(\EsC_A(y)-\sum_{z\in C^c} \Hm^{B}_\KRW(z,y)\EsC_A(z))
\geq\BCap(A)-\sum_{y\in A}\sum_{z\in C^c}\Hm^{B}_\KRW(z,y).\\
\end{align*}
As in the proof for the upper bound, we have:
\begin{align*}
\sum_{y\in A}&\sum_{z\in C^c}\Hm^{B}_\KRW(z,y)
\leq\sum_{y\in A}\sum_{z\in C^c}\sum_{w\in B}\theta(w-z)G_A(w,y)\\
=&\sum_{y\in A}\sum_{w\in B}G_A(w,y)\sum_{z\in C^c}\theta(w-z)
\stackrel{\eqref{ca2}}{\leq}\sum_{w\in B}\sum_{y\in A}G_A(w,y)\theta\{\left(\Ball((a-1)\rho)\right)^c\}\\
=&\theta\{\left(\Ball((a-1)\rho)\right)^c\}\sum_{w\in B}p_A(w)
\leq\theta\{\left(\Ball((a-1)\rho)\right)^c\}|B|
\stackrel{\eqref{ca}}{\leq}\frac{\BCap(A)}{2}.
\end{align*}
Now \eqref{b1} follows and this completes the proof of the lower bound.
\end{proof}

\section*{Acknowledgements}

We thank Professor Omer Angel for valuable comments on earlier versions of this paper.

\end{document}